\documentclass{article}
\usepackage[table]{xcolor}

\usepackage{microtype}
\usepackage{graphicx}
\usepackage{subcaption}
\usepackage{booktabs} 

\usepackage{hyperref}

\usepackage[preprint]{icml2026}

\usepackage{amsmath}
\usepackage{amssymb}
\usepackage{mathtools}
\usepackage{amsthm}
\usepackage{thmtools}
\usepackage{bbm}
\usepackage{dsfont}
\usepackage{setspace}

\usepackage[capitalize,noabbrev]{cleveref}

\usepackage{relsize}
\usepackage{enumitem}
\usepackage{subcaption}
\usepackage{float}
\usepackage{adjustbox}
\usepackage{arydshln}
\definecolor{cell-gray}{HTML}{d9dad8}
\def\thickhline{\noalign{\hrule height.8pt}}

\theoremstyle{plain}
\newtheorem{theorem}{Theorem}[section]

\newtheorem{lemma}[theorem]{Lemma}

\theoremstyle{definition}
\newtheorem{definition}[theorem]{Definition}

\theoremstyle{remark}
\newtheorem{remark}[theorem]{Remark}

\crefformat{equation}{#2Eq.~#1#3}
\Crefformat{equation}{#2Eq.~#1#3}

\AddToHook{cmd/appendix/before}{\crefalias{section}{appendix}}

\usepackage[textsize=tiny]{todonotes}
\newcommand{\R}{\mathbb{R}}
\newcommand{\Rset}{\mathbb{R}}
\newcommand{\Nset}{\mathbb{N}}

\newcommand{\loi}{\rho}

\newcommand{\dd}{\textrm{d}}
\newcommand{\divergence}{\operatorname{div}}
\renewcommand{\epsilon}{\varepsilon}
\renewcommand{\phi}{\varphi}
\newcommand{\defeq}{\stackrel{\mathsmaller{\mathsf{def}}}{=}}
\newcommand{\Var}{\operatorname{Var}}

\usepackage[normalem]{ulem}

\newcounter{marginNoteCounter}

\newcommand{\mySqBullet}		{\raisebox{0.25em}{{\scriptsize$_\blacksquare$}}}

\newcommand{\inlinetitle}[2]  {\smallskip\noindent\textbf{#1{#2}}}
\renewcommand{\paragraph}[1]  {\inlinetitle{#1}{}~}

\makeatletter
\AddToHook{cmd/appendix/before}{\def\cref@section@alias{appendix}}
\makeatother

\icmltitlerunning{Enhancing Exploration in Global Optimization}

\usepackage{xspace}
\newcommand{\MKV}  {MKV\xspace}  

\begin{document}

\twocolumn[
  \icmltitle{Enhancing Exploration in Global Optimization\\ by Noise Injection in the Probability Measures Space}



  \icmlsetsymbol{equal}{$\dagger$}

  \begin{icmlauthorlist}
    \icmlauthor{Gaëtan Serré}{cebo,equal}
    \icmlauthor{Pierre Germain}{ponts,swiss,equal}
    \icmlauthor{Samuel Gruffaz}{cebo,equal}
    \icmlauthor{Argyris Kalogeratos}{cebo}
  \end{icmlauthorlist}

  \icmlaffiliation{cebo}{Université Paris-Saclay, ENS Paris-Saclay, CNRS, Centre Borelli, 91190 Gif-sur-Yvette, France}
  \icmlaffiliation{ponts}{École des Ponts ParisTech, Marne-la-Vallée, France}
  \icmlaffiliation{swiss}{University of Neuch\^atel, Neuch\^atel, Switzerland}

  \icmlcorrespondingauthor{Pierre Germain}{pierre.germain@unine.ch}
  \icmlcorrespondingauthor{Gaëtan Serré}{gaetan.serre@ens-paris-saclay.fr}
   \icmlcorrespondingauthor{Samuel Gruffaz}{samuel.gruffaz@ens-paris-saclay.fr}

  \icmlkeywords{Optimization, McKean Vlasov, Probability, Sampling}

  \vskip 0.3in
]



\printAffiliationsAndNotice{\icmlEqualContribution}

\begin{abstract}
McKean-Vlasov (\MKV) systems provide a unifying framework for recent state-of-the-art particle-based methods for global optimization. While individual particles follow stochastic trajectories, the probability law evolves deterministically in the mean-field limit, potentially limiting exploration in multimodal landscapes. We introduce two principled approaches to inject noise directly into the probability law dynamics: a perturbative method based on conditional MKV theory, and a geometric approach leveraging tangent space structure. While these approaches are of independent interest, the aim of this work is to apply them to global optimization. Our framework applies generically to any method that can be formulated as a MKV system. Extensive experiments on multimodal objective functions demonstrate that both our noise injection strategies enhance consistently the exploration and convergence across different configurations of dynamics, such as Langevin, Consensus-Based Optimization, and Stein Boltzmann Sampling, providing a versatile toolkit for global optimization.
\end{abstract}

\section{Introduction}
Global optimization aims at finding the global minimum $x^\star$ of a non-convex function $V$ 
\cite{Pinter1996, lee2017finding, Awasthi2024, Houssein2024, el2024taxonomy, Franceschi2024}. While deterministic optimization methods, such as the ubiquitous gradient descent, are efficient for convex problems, they often get trapped in local minima when the function landscape is non-convex. A natural approach is to introduce randomness, which forces the optimization algorithm to explore the search space and escape suboptimal solutions.

In finite dimensions, a successful strategy combines randomness with stochastic dynamics
via the Langevin equation:
\begin{equation} \label{eq:langevin}
\dd X_t = -\nabla V(X_t) \dd t + \sqrt{2\kappa} \dd B_t,
\end{equation}
where $\kappa \geq 0$ controls the noise intensity and $B_t$ is a standard Brownian motion \cite{roberts1996langevin}. Under mild conditions, the distribution $\loi_t$ of $X_t$ converges to the Gibbs measure $\loi_\kappa(\dd x) \propto \exp(-V(x)/\kappa)$ related to potential $V$, which concentrates near global minima for appropriate choice of $\kappa$. Moreover, the evolution $t \mapsto \loi_t$ corresponds to the gradient flow of the free energy functional $\mathcal{F}_\kappa(\loi) = \mathrm{KL}(\loi \, \| \, \loi_\kappa)$ in Wasserstein space. Convergence to $\loi_\kappa$ requires either very large time or high temperature $\kappa$. However, increasing $\kappa$ to speed up convergence has a downside: it spreads $\loi_\kappa$ out rather than concentrating it on 
the set of global minimizers $X^\star$. To adaptively manage this trade-off, a scheduler for gradually decreasing the temperature is necessary to achieve both rapid convergence and concentration of the limiting measure on 
 $X^\star$. Simulated annealing has broadly exemplified this approach.

\paragraph{McKean-Vlasov systems.}
To balance exploration and exploitation more effectively, modern algorithms employ multiple interacting particles with different initializations. These methods, ranging from heuristic approaches to principled gradient flows on probability spaces, can be formulated as instances of McKean-Vlasov (\MKV) systems:
\begin{equation}\label{eq:mckean-vlasov}
\dd X_t^i = b(X_t^i, \hat{\loi}_{X_t}) \dd t + \sigma(X_t^i, \hat{\loi}_{X_t}) \dd B_t^i,
\end{equation}
where $\hat{\loi}_{X_t}$ is the empirical distribution of all $N$ particles at time $t$, and $(B_t^i)_{i \geq 1}$ are independent Brownian motions. The drift $b$ and diffusion $\sigma$ coefficients, which depend on the particle distribution, 
induce particle interactions that promote both exploration and exploitation. This general framework unifies several recently proposed methods: Consensus-Based Optimization (CBO) \cite{pinnau2017cbo} and Stein Boltzmann Sampling (SBS) \cite{serre2025stein} are both special instances of \cref{eq:mckean-vlasov} with specific choices of $b$ and $\sigma$. We refer the reader to \cref{appendix:background} for a detailed introduction to stochastic differential equations and \MKV theory.

In the mean-field limit as $N \to \infty$, the empirical distribution $\hat{\loi}_{X_t}$ converges to a deterministic probability law $\loi_t$, reducing the system to deterministic mean-field dynamics. The convergence of $\loi_t$ to a measure supported on $X^\star$ is then studied (e.g. \cite{koss2024mean, fornasier2024consensus, serre2025stein}), with the Gibbs measure being a canonical example \cite{hwang1980laplace}. However, this deterministic evolution of $\loi_t$, arising from averaging independent Brownian motions, potentially limits exploration at the population level, causing the dynamics to fail at exploring multiple modes of complex landscapes and to get trapped in suboptimal basins.

\begin{figure}[t]
\centering
\begin{subfigure}{0.23\textwidth}
    \includegraphics[width=\textwidth]{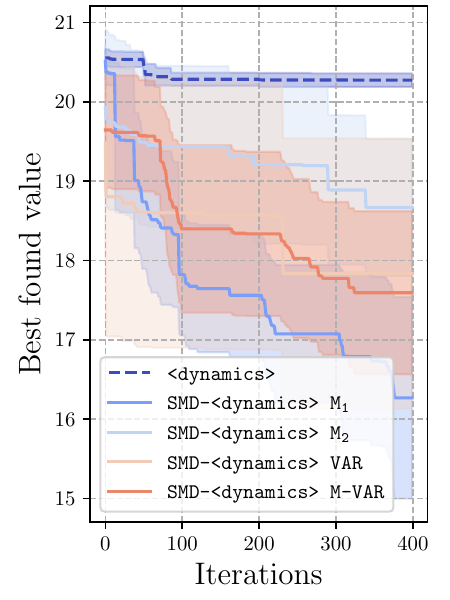}
    \caption{CBO dynamics}
\end{subfigure}
\hfill
\begin{subfigure}{0.23\textwidth}
    \includegraphics[width=\textwidth]{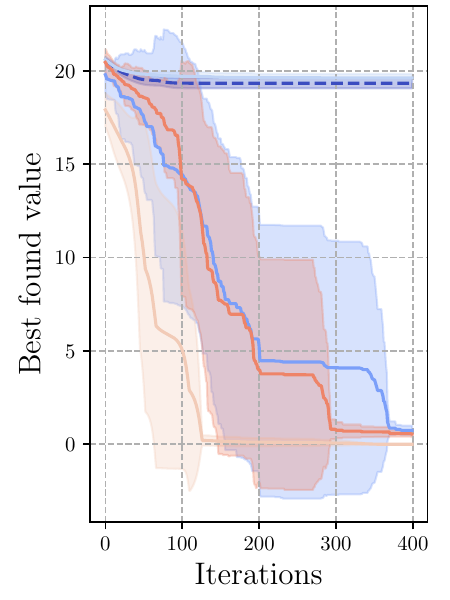}
    \caption{MSGD dynamics}
\end{subfigure}
  \caption{Impact of perturbed observables on CBO (left) and multi-start gradient descent dynamics (right) on Ackley function in dimension $4$ with $30$ particles. The mean (line) and standard deviation (shaded area) of the best found value are reported over $5$ runs.}
  \label{fig:impact_smd}
\end{figure}

\paragraph{Noise as a plug-in component.} To address this limitation, we adopt a distinctive modular perspective: we treat noise as a separate component that can be incorporated to any instance of \MKV dynamics. We develop two principled recipes for injecting noise directly into the probability law evolution itself, that we call $\rho$-\textit{noise} for the sake of conciseness and consistency. The first approach stems from \citet{germain2025stochastic}; it relies on Conditional \MKV theory and adds finite-dimensional common noise designed so that key observables (mean, variance) follow a prescribed stochastic evolution. This correlated motion enables collective transitions and facilitates the exploration of multiple basins of attraction. The second approach, novel to this work, leverages the geometric structure of probability spaces by introducing noise in the tangent space of the probability law manifold. This modular approach allows our noise injection strategies to be applied generically to any 
\MKV-based method, treating noise enhancement as a plug-in module that is independent of the underlying dynamics.

\begin{figure}[t]
\centering
\begin{subfigure}{0.23\textwidth}
    \includegraphics[width=\textwidth]{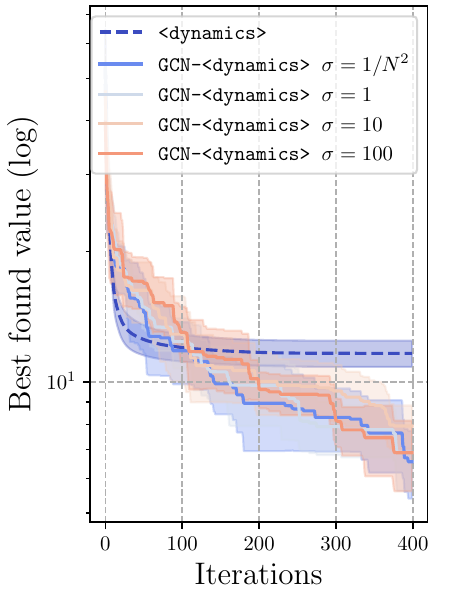}
    \caption{SBS dynamics}
\end{subfigure}
\hfill
\begin{subfigure}{0.23\textwidth}
    \includegraphics[width=\textwidth]{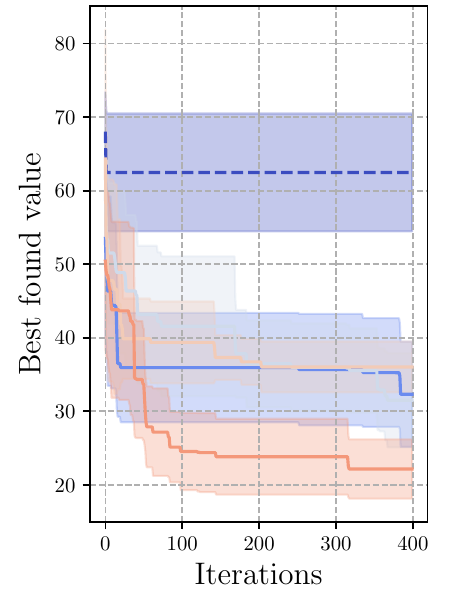}
    \caption{CBO dynamics}
\end{subfigure}
  \caption{Impact of $\sigma$ on the RKHS noise for the SBS dynamics (left) and the CBO dynamics (right) on the Levy function in dimension $10$ with $10$ particles. The mean (line) and standard deviation (shaded area) of the best found value are reported over $5$ runs.}
  \label{fig:impact_gcn}
\end{figure}

\paragraph{Contributions.}
Being the first to use $\rho$-\textit{noise} for global optimization, this paper makes the following main contributions:
\begin{itemize}[topsep=0pt, itemsep=0pt]
  \item[\mySqBullet] \textbf{Two principled approaches for injecting noise} into \MKV probability law dynamics: (i) an application of the perturbation-based method of \citet{germain2025stochastic}, namely \emph{Stochastic Moment Dynamics} (SMD), to $d$-dimensional observables; and (ii) \emph{Geometric Common Noise} (GCN), a novel geometric approach that exploits the tangent-space structure of probability measure spaces to induce Langevin-type dynamics at the level of probability laws.
	  
  \item[\mySqBullet] \textbf{Comprehensive experimental validation} across multiple \MKV dynamics (Langevin, Consensus-Based Optimization, Stein Boltzmann Sampling, and MSGD) and $7$ multimodal benchmark functions, showing that common noise substantially improves exploration and convergence compared to vanilla methods, opening new avenues for both global optimization and sampling applications.
\end{itemize}
\paragraph{Notations.}~The finite set $\{ 1, \dots, N \}$ is denoted by $[N]$. The pushforward of a measure is denoted by \#, i.e. $f\#\nu(\mathsf{A})=\nu(f^{-1}(\mathsf{A}))$ for any measurable function $f: (\mathcal{X},\mathcal{B})\to (\mathcal{Y},\mathcal{A})$, $\mathsf{A}\in \mathcal{A}$ ($\sigma$-algebra on $\mathcal{Y}$) and $\nu$ a measure on $(\mathcal{X},\mathcal{B})$. 
We denote by $k_\sigma^{(d)}(x,y):=\exp(-\|x-y\|_2^2/\sigma)$
the Gaussian kernel for any $(x,y)\in(\R^d)^2$ and $\sigma>0$. $\mathcal{S}_d^{++}$ is the set of $d$-square symmetric positive definite matrices.

\section{Preliminaries}
Recent advances in global optimization have been driven by \MKV-type algorithms, which represent the state-of-the-art for particle-based methods. These approaches leverage interacting particle systems that evolve according to mean-field dynamics, and enable effective exploration of complex multimodal objective functions.

In this work, we consider a modular view over \MKV dynamics, in which noise injection can be handled by a generic plug-in component, regardless the specifics of a method's dynamics. Next, we introduce key \MKV-based algorithms that have demonstrated superior performance on challenging optimization tasks and are candidates for incorporating our noise injection mechanisms. This diverse list also justifies the broad impact of our approach.

\mySqBullet~\textbf{MSGD $(b,\sigma)=(-\nabla V,0)$.} Multi-start gradient descent (MSGD) is a widely used baseline for global optimization, where multiple particles are initialized randomly and updated via gradient descent.

\mySqBullet~\textbf{Langevin $(b,\sigma)=(-\nabla V,\sqrt{2\kappa}I_d),\, \kappa >0$.} The Langevin dynamics is a stochastic optimization method that combines gradient descent with isotropic Brownian noise. 


\mySqBullet~\textbf{SBS $(b,\sigma)=(b^{\text{svgd}},0)$.} Stein Boltzmann Sampling (SBS) \cite{serre2025stein} is a particle-based optimization method that leverages repulsive forces $\nabla_x k(\cdot,x)$ between particles to enhance exploration, the drift term is defined as follows
\begin{equation}
\label{eq:sbs_drift}
    b^{\text{svgd}}(\cdot,\rho) 
\defeq  \mathbb{E}_{x\sim\loi}
\Big[
-k(\cdot,x)\,\nabla V (x)
+ \kappa \nabla_x k(\cdot,x) \Big]
\end{equation} where $\kappa>0$ is a temperature parameter and $k$ is a positive definite kernel, typically chosen as a Gaussian kernel. 
It is based on Stein Variational Gradient Descent (SVGD), which is originally a sampling method derived from the gradient flow of the free energy functional $\mathcal{F}_\kappa$ in the Stein geometry framework \cite{duncan2023geometry}. In \cref{sec:GCN}, we build on this observation to design a $\rho$-\textit{noise} that we believe to be canonical.

\mySqBullet~\textbf{CBO} $(b,\sigma)=(b^{\text{CBO}},\sigma^{\text{CBO}})$\textbf{.} Consensus-Based Optimization (CBO) \cite{pinnau2017cbo} is a popular particle-based optimization method that drives particles towards a consensus point computed as a weighted average of the best-performing particles. The resulting \MKV system has drift and diffusion terms given respectively by:%
\begin{align*}
b^{\text{CBO}}(X^i_t, \hat{\loi}_{X_t}) &\defeq -\lambda(X^i_t - v_f ) H^\epsilon(f (X^i_t) - f (v_f))
\\
\sigma^{\text{CBO}}(X^i_t, \hat{\loi}_{X_t}) &\defeq \gamma \|X^i_t - v_f\|_2 I_d,
\end{align*}%
where $H^\epsilon$ is a smooth approximation of the Heaviside function, $\lambda, \gamma > 0$ are hyperparameters, and $v_f$ is the consensus point defined as:
$$
v_f \defeq \frac{\sum_{i=1}^N X_t^i e^{-\alpha f(X_t^i)}}{\sum_{i=1}^N e^{-\alpha f(X_t^i)}},
$$
where $\alpha > 0$ is a weighting parameter corresponding to the inverse temperature in a Gibbs measure. While anisotropic noise variants of CBO have been proposed \cite{carrillo2021cbo} that could extend this list, we focus on the original isotropic formulation for straightforward compatibility with the \MKV framework.

\section{Methodology}
First, we present an ad hoc finite-dimensional $\rho$-\textit{noise}, called SMD, introduced in \cite{germain2025stochastic} (\Cref{sec:SMD}). We then introduce an infinite-dimensional $\rho$-\textit{noise}, which we claim is canonical when the MKV equation is derived from a gradient flow (\Cref{sec:GCN}). 
\subsection{Noise through moment perturbation}\label{sec:SMD}

\paragraph{Mean-field optimization and particle systems.}
Particle-based optimization methods \cite{pinnau2017cbo, serre2025stein, roberts1996langevin} can be viewed as evolutions in the space of probability measures, approximated by interacting particle systems. When particles are driven only by independent noise, stochasticity vanishes in the mean-field limit due to propagation of chaos, often resulting in poor exploration and convergence to local minima \cite{tugaut2014phase, monmarche2025local}.

\paragraph{Common noise and conditional \MKV theory.}
A natural remedy is to introduce a \emph{common noise} that acts on all particles simultaneously. Unlike independent noise, common noise persists in the mean-field limit and induces `genuinely' stochastic dynamics for the limiting distribution. This setting, known as \emph{conditional McKean–Vlasov theory} \cite{carmona_delarue_book2}, has recently been used to improve exploration \cite{delarue2025exploration}, escape local minima \cite{delarue2024ergodicity, maillet2023longtime_commonnoise}, and design structured perturbations \cite{delarue2025rearranged}. It is particularly motivated by two failure modes: particle collapse (vanishing variance) and collective trapping in suboptimal basins.

\paragraph{Macroscopic perturbations via stochastic moments.}
To address these issues, one can perturb directly macroscopic observables (e.g. the empirical mean or variance). This is achieved by augmenting the mean-field dynamics with a suitably designed common noise so that the induced evolution of these moments remains stochastic in the mean-field limit. The resulting $\rho$-noise, termed \emph{Stochastic Moment Dynamics} (SMD), was introduced in \citet{germain2025stochastic} for general observables, together with a theoretical analysis.
Here, we provide a self-contained presentation of the method, with a focus on practical implementation.

\paragraph{Macroscopic observable of interest $F$. }
Let $N \ge 1$, starting from the baseline mean-field particle dynamics \cref{eq:mckean-vlasov}, choose a macroscopic observable $F:\mathcal{P}(\R^d)\rightarrow \R^p$ from:
\begin{itemize}[topsep=0pt, itemsep=0pt]
    \item[\mySqBullet] Mean: $F_{\text{mean}}(\rho)\defeq \int_{\R^d}xd\rho(x)$
    \item[\mySqBullet] Second-order moment: $F_{\text{s.m}}(\rho) \defeq \left( \int_{\R}x^2d\rho_j(x) \right)_{j\in [d]} $
    \item[\mySqBullet] Variance: $F_{\text{var}} \defeq \left( \Var(\rho_j) \right)_{j\in [d]} $ 
    \item[\mySqBullet] Mean + Variance : $[F_{\text{mean}}(\rho), F_{\text{var}}( \rho)]$.
\end{itemize}
Above, $\rho \in \mathcal{P}(\R^d)$, $\rho_j $ is the $j$-th marginal of $\rho$, and $p\in \{d,2d\}$ depends on the choice of the observable. 

\paragraph{SMD General Recipe. }
The method consists in adding a stochastic perturbation in \cref{eq:mckean-vlasov} according to the following equation:
\begin{equation}
\begin{aligned}
\label{eq:particle_smd_continuous_icml}
\mathrm{d}X_t^i
={}
 &b(X_t^i, \hat{\loi}_{X_t})\,\mathrm{d}t
+
\sigma(X_t^i, \hat{\loi}_{X_t})\,\mathrm{d}B_t^i \\
&+ \beta \, \big[  
\tilde{b}(X_t^i, \hat{\loi}_{X_t})\,\mathrm{d}t
+
\tilde\sigma(X_t^i, \hat{\loi}_{X_t})\,\mathrm{d}B_t\big].
\end{aligned}
\end{equation}

where $\beta\geq0$ is an intensity parameter and $B_t$ is a $p$-dimensional Brownian motion acting on all particles. 
The resulting dynamics \cref{eq:particle_smd_continuous_icml} can thus be viewed as a combination of two effects: a deterministic mean-field drift—associated to $b$ and $\sigma$ (driving the system toward minimizers of the objective function), and a
stochastic perturbation (associated to $\tilde{b}$ and $\tilde{\sigma}$) providing macroscopic exploration. The intensity $\beta$ of this additional exploration perturbation can be tuned in close analogy with the role of temperature in the finite-dimensional Langevin dynamics \cref{eq:langevin}.

 Consider independently the additional perturbation driven by $\tilde b$ and $\tilde \sigma$:
\begin{equation}
\label{eq:forcing}
\mathrm{d}Y_t^i
=
\tilde{b}(Y_t^i, \hat{\loi}_{Y_t})\,\mathrm{d}t
+
\tilde\sigma(Y_t^i, \hat{\loi}_{Y_t})\,\mathrm{d}B_t .
\end{equation}
The coefficients $\tilde{b}$ and $\tilde \sigma$ are defined so that the chosen macroscopic quantity of $\hat{\loi}_{Y_t}$ , independently from the number $N$ of particles in the system, solves the stochastic differential equation given by some coefficients $a\in \R^p$ and $s\in \R^{p\times p}$:
\begin{equation} \label{eq:controlemoment}
\mathrm{d}F(\hat{\loi}_{Y_t})
=
a(F(\hat{\loi}_{Y_t}))\,\mathrm{d}t
+
s(F(\hat{\loi}_{Y_t}))\,\mathrm{d}B_t .
\end{equation}

The expression of $\tilde{b}$ and $\tilde \sigma$ is given by an explicit formula from \citet{germain2025stochastic} and only depends on $F$, $a$ and $s$. Therefore, given a choice of a macroscopic quantity $F$, a critical aspect of the method lies in the choice of the coefficients $a$ and $s$. When $F$ is positive, $a,s$ are chosen such that \cref{eq:controlemoment} is a Bessel process in such way that $F(\hat{\loi}_{Y_t})$ stay positive, i.e.
$$a_{+}^\delta(z)= \left(\frac{\delta- 1 / 2 }{z_i} \right)_{i\in [d]}  , \quad s_+^\delta(z)=I_{d}$$
where $\delta\geq2$ is a parameter whose role will be explained later. The formula of $(\tilde b,\tilde \sigma)$ are given for any choice of perturbation $F\in \{F_{\text{mean}},F_{\text{s.m}},F_{\text{var}},[F_{\text{mean}},F_{\text{var}}] \}$ in the following lemma.

\begin{lemma}
\noindent\textbf{SMD-Mean $F=F_{\text{mean}}$, $(a,s)=(0_d,I_d)$.}
\begin{equation} \label{eq:mean}
    \tilde b(x,\loi) \defeq 0  , \; \; \tilde \sigma(x,\loi) \defeq I_d,    
\end{equation}

\noindent\textbf{SMD Second-order moment $F=F_{\text{s.m}}$, $(a,s)=(a_+^\delta,s_+^\delta)$.}
\begin{equation} \label{eq:moment2}
\begin{aligned}
\tilde b(x,\loi) &\defeq \left( \delta-\frac{3}{2} \right)\, \left[ \frac{x}{4m_2(\rho)^2} \right],\\
\tilde \sigma(x,\loi) &\defeq  \text{Diag} \left[ \frac{x}{2m_2(\rho)} \right] .
\end{aligned}
\end{equation}

\noindent\textbf{SMD Variance $F=F_{\text{var}}$, $(a,s)=(a_+^\delta,s_+^\delta)$. }
\begin{equation} \label{eq:variance}
\begin{aligned}
\tilde b(x,\loi) &\defeq \left( \delta-\frac{3}{2} \right)\, \left[ \frac{x-m(\rho)}{4\Var(\rho)^2} \right],\\
\tilde \sigma(x,\loi) &\defeq  \text{Diag} \left[ \frac{x-m(\rho)}{2\Var(\rho)} \right] .
\end{aligned}
\end{equation}

\noindent\textbf{SMD (Mean + variance) $F=[F_{\text{mean}},F_{\text{var}}]$, $(a,s)=((0_d,a_+^\delta),(I_d,s_+^\delta))$.} 
\begin{equation} \label{eq:meanvariance}
\begin{aligned}
\tilde \sigma(x,\loi) &\defeq \left( I_d \, , \, \text{Diag} \left[ \frac{x-m(\rho)}{2\Var(\rho)} \right]  \right),
\end{aligned}
\end{equation}
where $\tilde b(x,\loi)$ is the same of the variance perturbation case.
\end{lemma}
The derivation of the coefficients $\tilde b$ and $\tilde \sigma$ is detailed in \cref{appendix:derivationcoeffs}, which were given only for dimension $d=1$ in \citet{germain2025stochastic}. The empirical efficiency of these methods is tested on a benchmark of multi-modal functions across multiple MKV dynamics in \cref{sec:numerics}. Depending on the nature of the objective function to minimize, one should adapt its choice of perturbed macroscopic observable to maximize the efficiency of the method, even though Mean+Var will be recommended as a practical default.

\paragraph{Positivity constraints and Bessel dynamics.}
For a given macroscopic observable $F$, poorly chosen coefficients $a$ and $s$ may cause the dynamics~\cref{eq:forcing} and \cref{eq:particle_smd_continuous_icml} to blow up in finite time. This issue arises because many macroscopic observables (e.g., variances) take values in a restricted subset of $\mathbb{R}^p$. As long as the stochastic perturbation of $F(\hat{\loi}_{Y_t})$ remains within its admissible domain, the dynamics are well defined; however, if positivity constraints are violated, the coefficients become singular and the evolution breaks down.

To enforce positivity when perturbing second-order moments or variances, the coefficients $a$ and $s$ are chosen as $d$ independent $\delta$-Bessel processes \citep[Chapter~11]{revuz2013continuous}, which preserve positivity for $\delta \geq 2$. These correspond to the SDE~\cref{eq:controlemoment} with $(a,s)=(a_+^\delta,s_+^\delta)$, understood component-wise. In contrast, since the mean is unconstrained, its perturbation is taken to be standard Brownian motion ($a=0$, $s=I_d$).

The coefficients defined in \cref{eq:moment2}, \cref{eq:variance}, and \cref{eq:meanvariance}, are not globally Lipschitz in $\rho$, reflecting potential singular behavior as second-order moments or variances approach zero. While this could in principle lead to finite-time explosion, \citet[Section~6]{germain2025stochastic} show that the stochastic perturbation alone is non-explosive. It is therefore conjectured that the full dynamics~\cref{eq:particle_smd_continuous_icml}, combining this perturbation with a broad class of mean-field optimization algorithms, is also non-explosive under mild conditions. Since no explosion was observed in our numerical experiments, we do not consider regularized dynamics in the remainder of this paper.

\subsection{Noise using tangent space}
\label{sec:GCN}




The second approach relies on the geometric structure of the considered probability law space related to a dynamic. We introduce an infinite-dimensional noise that we claim to be canonical by injecting noise in tangent spaces related to the curve $t \mapsto \loi_t$. For the presentation of the approach, we follow the case-study of Stein variational gradient descent (SVGD). As Langevin, SVGD is a sampling method that can be used for global optimization \cite{serre2025stein}.
The recipe is summarized at the end of the section, the simulation scheme is detailed in \cref{appendix:GCN_scheme} and all the proofs are given in \cref{appendix:proof}.

\paragraph{SVGD mean-field dynamics. }
Given a kernel $k:(x,y)\in \mathbb{R}^d\times \Rset^d\mapsto k(x,y) $, typically a Gaussian kernel $k=k_\sigma^{(d)}$, the SVGD's $N$-particle dynamic is a \MKV dynamics using $b^{\text{svgd}}$ \cref{eq:sbs_drift} as drift and with zero diffusion.

When $N\to \infty$, and under mild conditions, $\hat{\loi}_{X_t} \to \loi^{\text{svgd}}_t$ following the mean-field dynamics:
\begin{equation}
\begin{aligned}
\label{eq:gradient_flow_svgd}
    &\partial_t \loi_t^{\text{svgd}} =-\divergence\left(b^{\text{svgd}}(\cdot,\rho_t^{\text{svgd}}) \loi_t^{\text{svgd}}\right)
\end{aligned}
\end{equation}
As Langevin dynamics, when $t\to \infty $, $\rho_t^{\text{svgd}}$ is proved to converge to the Gibbs measure $\mu(x)\propto \exp(-V(x)/\kappa)$ related to potential $V$. In the following, we set the temperature parameter $\kappa=1$.


In the following, we detail the structure of \cref{eq:gradient_flow_svgd} to 
clarify the $\rho$-\textit{noise}.

\paragraph{Gradient flow on Wasserstein space.}
\cref{eq:gradient_flow_svgd} is in fact a gradient flow related to a specific geometry on the space of probability law \cite{duncan2023geometry}. To give an intuition on this interpretation, note that a sum of two probability measures $(\loi,\nu)$ is not a probability measure $\loi+\nu$ since $(\loi+\nu)(\Rset^d)=2$. Instead, 
due to mass conservation, the following property holds: if a curve $t\in [0,1] \to \loi_t\in \mathcal{P}_2=\{\loi\in \mathcal{P}(\Rset^d): \mathbb{E}_{\loi}\left[|X|^2 \right]<\infty \}$ is absolutely continuous, then it satisfies the continuity equation in the sense of distribution \citep[Theorem 8.3.1]{ambrosio2005gradient}, i.e. there exists $v_t\in L^2(\loi_t) $ for any $t\in[0,1]$ such that
\begin{equation}
\label{eq:continuity_equation}
    \partial_t \loi_t+\divergence(\loi_tv_t)=0 \, .
\end{equation}
This means that an infinitesimal displacement around $\loi_t$ has the form $\divergence(\loi_tv_t)$. 

Note that the discretization of \cref{eq:continuity_equation} can be derived using:
\begin{equation}
\label{eq:discretization_eq_con}
\begin{aligned}
    &\hat{\loi}_{(n+1)\epsilon}=(id+\epsilon v_{n\epsilon})\#\hat{\loi}_{n\epsilon},\\
    \text{since}\qquad &\lim_{\epsilon\to 0} \frac{(id+v_t)\#\loi_{t}-\loi_t}{\epsilon}=\divergence(\loi_tv_t) .
\end{aligned}
\end{equation}
For these reasons, the functional space where $v_t$ belongs can be identified as a tangent space using Riemannian geometry vocabulary, and by using the pushforward operation `$\#$' instead of the addition `$+$'. This statement is made rigorous in \citet[Theorem 8.5.1]{ambrosio2005gradient}. 

In general, the tangent space is as subspace of:
$$
L^2(\loi_t)=\left\{ f:\R^d\mapsto \R^d\, \text{measurable}:\int \|f\|^2\dd \loi_t<\infty \right\}.
$$
In the case of SVGD dynamics, the functional space is related to a subspace of a Reproducible Kernel Hilbert Space (RKHS) $\mathcal{H}_d\subset L^2(\loi_t)$ generated by $K_{SVGD}=kI_d$ \cite{duncan2023geometry}.




Therefore, to \emph{inject canonical noise} in a trajectory of probability laws $(\loi_t)$, we only have to inject canonical noise in the tangent space where $v_t$ belongs according to its structure,
\begin{align}
\label{eq:discretization_eq_con_prob}
    &\hat{\loi}_{(n+1)\epsilon}=(id+\epsilon v_{n\epsilon}+\sqrt{\epsilon}G_n)\#\hat{\loi}_{n\epsilon},\quad G_n \sim \mathcal{G}.
\end{align}
The question is then, how to add canonical noise in the tangent space such that its time increments $(G_n)_{n\in \mathbb{N}}\sim \mathcal{G}$ act as an additive noise in equation \cref{eq:discretization_eq_con_prob}? Can we characterize their law $\mathcal{G}$, as well as the law $\rho^*_\cdot$ of the resulting continuous time noising process? 


\paragraph{Canonical noise in RKHS. }
 Brownian motion, or any Gaussian process, is a random variable in an Hilbert space of functions typed as $\Rset\to \Rset$. Its natural generalization to Hilbert spaces of functions $\R^d\rightarrow \R^d$ is called a $d$-dimensional space Gaussian process, 
or Gaussian random field \cite{estrade2025gaussian,mackay1998introduction}. In the following, we focus on RKHS-valued Gaussian random fields.


\begin{definition}[$\mathcal{H}_d$-Gaussian random field]
A random variable $G:\Omega\to\mathcal H_d$ is called a (centered) $\mathcal{H}_d$-Gaussian random field if, for every $f\in\mathcal H_d$, the real-valued random variable
$\langle G,f\rangle_{\mathcal H_d}$ is a centered Gaussian random variable on $\mathbb R$.
Moreover, if $\mathbb E\|G\|_{\mathcal H_d}^2<\infty$, then
its law is characterized by a covariance kernel $K:\Rset^d\times \Rset^d \mapsto \mathcal{S}_d^{++}\subset \Rset^{d\times d} $ and we denote $G\sim \mathcal{G}(K)$.  
More precisely, for any $X=(x_i)_{i\in[N]}\in (\Rset^d)^N$, $(G(x_i))_{i\in [N]}$ is a $\R^{Nd}$ Gaussian vector, whose covariance matrix $\bar K(X)=(K_{i,j}(X))_{i,j\in [N]} \in \R^{Nd\times Nd}$ is given by the blocs:
\begin{equation*}
K_{i,j}(X)=K(x_i,x_j)=\mathbb{E}[G(x_i) G(x_j)^T] \in \R^{d\times d}.
\end{equation*}


\end{definition}

Heuristically, a centered Gaussian process $G$ with covariance kernel $\bar K$ can be interpreted
as a standard Gaussian on the RKHS $\mathcal H_d$, formally associated with density
$\exp(-\tfrac12\|G\|_{\mathcal H_d}^2)$. This interpretation cannot be made rigorous in
$\mathcal H_d$, since infinite-dimensional Hilbert spaces do not admit a Lebesgue reference
measure. A rigorous meaning is obtained by restricting $G$ to finite-dimensional subspaces
$\mathcal H_d^X$ generated by kernel evaluations:
$$
    \mathcal{H}_d^X=\operatorname{span}(K(x_i,\cdot)e_j,\, i,j\in[N]\times [d])\subset \mathcal{H}_d \, .
$$
by considering its orthogonal projection $\pi_XG$ into $\mathcal{H}_d^X$. 

\begin{lemma} \label{lemma:reduced_N}
Let $X=(x_1,\dots,x_N)$ with $x_i\ne x_j$ if $i\ne j$, and let $G\sim \mathcal{G}(K_{SVGD})$. Then the projection
$\pi_X G\in\mathcal H_d^X$ admits a density $\mathbb P(\pi_X G\in \mathrm df)$ with respect to the Lebesgue measure induced by the
RKHS inner product, given by
\[
\mathbb P(\pi_X G\in \mathrm df)\propto
\exp\!\big(-\tfrac12\|f\|_{\mathcal H_d}^2\big).
\]
Equivalently, $G(X)=(G(x_1),\dots,G(x_N))\in\mathbb R^{Nd}$ is a centered
Gaussian vector with density proportional to
$\exp(-\tfrac12\, g^\top \bar K(X)^{-1} g)$ with respect to the usual Lebesgue measure on $\R^{Nd}$.
\end{lemma}

All theoretical details are deferred to \cref{appendix:lemmedensite}.

 
\cref{lemma:reduced_N} justifies that $G\sim\mathcal{G}(K_{SVGD})$ is Gaussian standard in $\mathcal{H}_d$, and thus is considered canonical.

The following remark justifies that $d$-dimensional space Gaussian process is an infinite dimensional noise.
\begin{remark}
\label{remark:decomposition_gcn}
\onehalfspacing
When the kernel $k(x,y)=k_\sigma^{(d)}(x,y)$ is Gaussian, $G \defeq (G_i)_{i\in[d]}\sim \mathcal{G}(K_{\text{SVGD}})$ can be simulated as follows \citep[Eq (1.1.11)]{lototsky2017stochastic}
$$
G_i=\sum_{\vec{n}\in \Nset^d} \lambda_{\vec{n}} \phi_{\vec{n}}(\cdot) \xi_{\vec{n}}^i , \; \xi_{\vec{n}}^i\overset{i.i.d}{\sim}\mathcal{N}(0,1),\,  \vec{n}\in \mathbb{N}^d
$$
where $(\Phi_{\vec{n}}=\sqrt{\lambda_{\vec{n}}}\phi_{\vec{n}})_{\vec{n\in \Nset^d}} $ is an orthonomal basis of the RKHS related to $k_\sigma^{(d)}$ \cite{berlinet2011reproducing}.
\end{remark}
The take-home message is the following : if we aim to inject canonical noise in the SVGD, we should add time-dependent noise in \cref{eq:continuity_equation}, whose time increments are i.i.d $\mathcal{H}_d$-Gaussian random field of covariance kernel $K_{SVGD}$, i.e. $\mathcal{G}=\mathcal{G}(\beta^2 K_{SVGD})$ in \cref{eq:discretization_eq_con_prob} where $\beta>0$ is a temperature parameter. 

\paragraph{Continuous Langevin-type dynamics. }
Our goal is to study the time-continuous dynamics related to in \cref{eq:discretization_eq_con_prob} when injecting canonical noise $\mathcal{G}=\mathcal{G}(\beta^2 K_{SVGD})$, 
to get a Langevin-type dynamic on the space of probability laws, corresponding to a gradient flow \cref{eq:continuity_equation} with a canonical $\rho$-\textit{noise}. Since the continuity equation is a partial differential equation, its stochastic
counterpart naturally takes the form of a stochastic partial differential equation (SPDE).
Accordingly, the noise must be modeled as a Gaussian space-time process
$(G_t)_{t\ge0}\sim\mathcal{TG}(K_{SVGD})$ with increments
$G_{t+\varepsilon}-G_t\sim\mathcal{G}(\varepsilon K_{SVGD})$), which can be viewed as the infinite-dimensional analogue of Brownian motion. In the SPDE, this leads to the appearance of the infinitesimal noise term $dG_t$, interpreted as Gaussian white noise, that is, the distributional time derivative of $G_t$ \citep[Eq.~(1.1.5)]{lototsky2017stochastic}; see \cref{appendix:random_generalized_function} for an introduction to these objects.

Following \cref{remark:decomposition_gcn}, in the case of SVGD, the infinitesimal variation of $(G_t)_{t\geq0}=(G^i_t)_{i\in [d],t\geq0}\sim \mathcal{TG}(K_{\text{SVGD}})$ can be represented as follows,
    \begin{equation*}
    \dd G_i(t,\cdot)=\sum_{\vec{n}\in \Nset^d} \lambda_{\vec{n}} \phi_{\vec{n}}(\cdot) \dd \xi_{\vec{n},t}^i ,
\end{equation*}
where $(\xi_{\vec{n},t}^i)_{t\geq 0}$ are independent Brownian motions \citep[Eq (1.1.10)]{lototsky2017stochastic}.


To define a valid SPDE, the idea is first to consider the simplified case where $\rho_0=\hat\rho_0$ is the empirical measure of a particle system. Because the dynamic of the system is fully described by the motion of the $N$ particles, the SPDE can be rewritten as an SDE governing the evolution the particles. The initial SPDE is then obtained as a limit of the considered SDEs as $N\to \infty$. The SDE describing the motions of $N$ particles is derived from \cref{eq:noise_svgd}.

\begin{lemma}
\label{lemma:SDE_to_consider}
    The SDE being the time-continuous/discretized in space version of \cref{eq:discretization_eq_con_prob} with $\mathcal{G}=\mathcal{G}(\beta^2 K_{SVGD})$ and $\beta>0$ is
\begin{equation} \label{eq:noise_svgd}
    \dd \bar{X}_t= b^{\text{svgd}}\left(X_t,\loi_t^N \right)\dd t+\beta \bar{K}^{1/2}(X_t) \dd B_t^{Nd},
\end{equation}
with $\rho_t^N=\frac{1}{N}\sum_{i=1}^N\delta_{X_t^i}$ is the empirical measure of the particle system, $(B_t^{Nd})_{t\geq0}$ is a $Nd$ dimensional Brownian motion and $\bar{K}(X) = K(X)\otimes I_d \in \Rset^{dN\times dN} $,  $K(X)\defeq (k(X^i,X^j))_{i,j\in [N]}$. 
\end{lemma}

In the following, the term:
\begin{equation} \label{eq:G_t_N_bar}
    \bar{K}(X_t)^{1/2}\dd B_t^{Nd}
\end{equation}
will be referred as \textit{Geometrical Common Noise} (GCN). This diffusion coefficient is the consequence of the addition of the $\sqrt{\epsilon} G_n$ term in  \cref{eq:discretization_eq_con_prob}.

Surprisingly, when $\beta=\beta_N \defeq \sqrt{2}/\sqrt{N}$ the dynamic \cref{eq:noise_svgd} was already studied in the literature under the name of Stochastic Langevin with repulsive force; see \cite{gallego2018stochastic,nusken2021stein} and \citep[Section 3]{duncan2023geometry}.  In this setting the noise vanishes at the large $N$ limit and the evolution of $\rho_t^N$ becomes deterministic. In this paper, under the scaling $\beta_N=\beta>0$, the noise \cref{eq:G_t_N_bar} doesn't vanish when $N\rightarrow \infty$, thus GCN is a $\rho$-\textit{noise}. 

In the following theorem $t\to \rho_t^*$ is the noised dynamic $t\to \rho_t^{\text{svgd noisy}}$ that we look for.
\begin{theorem}
\label{thm:limit-theorem}
    If the family $\left\{\rho^N_\cdot: N \in \mathbb{N}\right\}$ possesses a limit point in $\mathcal{P}(C[0, T])$ denoted by $\rho^*_\cdot$, then $\rho^*_\cdot$ is solution in the weak sense of  
    \begin{equation}
    \begin{aligned}
    \label{eq:spde_noisy}
    \partial_t \rho_t^* +&\divergence \left(\left(b(\cdot,\rho_t^*) +\beta \dd G_t)\right)\rho_t^* \right)  \\
    & +\beta^2 \Delta_t (k(\cdot,\cdot) \rho_t^*) =0,\quad G\sim \mathcal{TG}(K_{\textup{SVGD}})
  \end{aligned}  
  \end{equation}
Moreover, in case where $k(x,x)=1$ for any $x\in \Rset^d$, which is verified when $K_{\text{SVGD}}=k_\sigma^{(d)}I_d$,
$\rho_t^*$ can be approximated using that
\begin{equation}
\label{eq:schem_info_svgd}
\begin{cases}
    \rho_{t+\epsilon}^* \approx (id+T_t)\# \rho_{t}^*, \;  g_t\sim\mathcal{G}(K_{\textup{SVGD}}) \\
    T_t \ \ \ \, = \epsilon b(\cdot,\rho_t^*)+\sqrt{\epsilon}\beta g_t\, .
\end{cases}.
\end{equation}
\end{theorem}
In practice, only the dynamics of $N$ particles is simulated using \cref{eq:noise_svgd}, whose the Euler-Maruyama integration scheme is given with details in \cref{appendix:GCN_scheme}. In our experiments the kernel $k$ is Gaussian. The scheme given in \cref{eq:schem_info_svgd} is here only to 
give more intuition on $\dd G$ in \cref{eq:spde_noisy}, which becomes $\beta g_t$ in \cref{eq:schem_info_svgd}, accounting the stochastic limit-effect of the infinitesimal $d$-dimensional space Gaussian process noise that was injected in \cref{eq:noise_svgd}.

By seeing $t\mapsto \rho_t^{\text{svgd}}$ as a gradient flow on a Wasserstein space \cite{duncan2023geometry}, the dynamic $t\mapsto \rho_t^*$ given in \cref{thm:limit-theorem} can be seen as a Langevin dynamic on the space of probability law related to the metric given in \citet[Definition 7]{duncan2023geometry}, that is why we consider the noise \cref{eq:G_t_N_bar} as canonical compared to SMD (\cref{sec:SMD}).

\paragraph{Generalization of the approach.}
In general, the kernel related to SVGD and the Gaussian process covariance function could be chosen differently, and the noise given by \cref{eq:G_t_N_bar} can be added to any \MKV dynamics, the limit-effect will keep the same structure than for SVGD. However, GCN cannot be considered as canonical if there is no link to the gradient flow structure of the original dynamic. 

In any case, the choice of the kernel should encode the regularity prior on the objective function. In experiments, the effect of the bandwidth $\sigma$ is studied when the covariance function is Gaussian ($k_\sigma I_d$). 

\paragraph{Link between GCN and SMD-mean. }
 The following Theorem shows that GCN is an interpolation between an infinite-dimensional white noise and SMD-mean, which is a $d$-dimensional noise, depending on the bandwidth parameter $\sigma$ when $k=k_\sigma^{(d)}$.
\begin{theorem}
\label{theorem:link}
    If the covariance function $k_\sigma^{(d)}$ is Gaussian: 
    \begin{itemize}[topsep=0pt, itemsep=0pt]
        \item If the parameter $\sigma \to \infty$, GCN is SMD-mean.
        \item If the parameter $\sigma \to 0$, GCN is a white noise.
    \end{itemize}
\end{theorem}




\newcommand{\tabsize}{{0.8\linewidth}}
\begin{table*}[t]
  \caption{Comparison between vanilla dynamics and their SMD and GCN variants on $7$ multimodal benchmark functions in dimension $20$ using $150$ particles over $300$ iterations. For each dynamic, the average best result is reported. The average rank over all benchmarks along with the ECR is reported at the bottom of each table. The maximum p-value of Mann-Whitney U tests is reported in the last column. The best results are in bold and cells where the Mann-Whitney U test is significant at level $0.05$ are highlighted in gray.}
  \label{tab:multimodal-results}
    \centering
    \begin{minipage}{0.05\textwidth}\vspace{0pt}
      \rotatebox{90}{\footnotesize CBO dynamics}
      \vspace*{-1em}
    \end{minipage}%
    \begin{minipage}{\tabsize}\vspace{0pt}
      \begin{adjustbox}{width=\linewidth}
        \begin{tabular}{l|cccccc|c}
        \thickhline
        \textbf{Benchmark} & \textbf{CBO} & \textbf{SMD-CBO Mean} & \textbf{SMD-CBO} $\mathbf{M^2}$ & \textbf{SMD-CBO Var} & \textbf{SMD-CBO Mean+Var} & \textbf{GCN-CBO} & \textbf{p-value} \\
        \hline
        Ackley & $21.097$ & $20.615$ & $21.083$ & $21.088$ & $20.610$ & $\cellcolor{cell-gray} \mathbf{20.591}$ & $0.000$ \\
        Deb N.1 & $\cellcolor{cell-gray} \mathbf{-1.000}$ & $\cellcolor{cell-gray} \mathbf{-1.000}$ & $\cellcolor{cell-gray} \mathbf{-1.000}$ & $\cellcolor{cell-gray} \mathbf{-1.000}$ & $\cellcolor{cell-gray} \mathbf{-1.000}$ & $\cellcolor{cell-gray} \mathbf{-1.000}$ & $0.000$ \\
        Griewank & $20.279$ & $19.562$ & $19.657$ & $19.683$ & $20.073$ & $\mathbf{19.506}$ & $0.237$ \\
        Levy & $102.351$ & $86.857$ & $95.816$ & $95.945$ & $\cellcolor{cell-gray} \mathbf{81.912}$ & $85.146$ & $0.000$ \\
        Rastrigin & $253.358$ & $238.943$ & $250.021$ & $250.463$ & $240.406$ & $\cellcolor{cell-gray} \mathbf{238.768}$ & $0.000$ \\
        Schwefel & $5157.189$ & $5127.111$ & $5139.564$ & $5146.530$ & $5199.557$ & $\mathbf{5111.855}$ & $0.387$ \\
        Styblinski-Tang & $-19.513$ & $\cellcolor{cell-gray} \mathbf{-20.798}$ & $-20.006$ & $-19.881$ & $-20.091$ & $-20.191$ & $0.000$ \\
        \cdashline{1-8}[2pt/1pt]
        Avg Rank & $5.14$ & $2.00$ & $3.29$ & $4.14$ & $3.00$ & $\mathbf{1.29}$  \\
        ECR & $1.0648$ & $\mathbf{1.0097}$ & $1.0424$ & $1.0443$ & $1.0132$ & $1.0104$ \\
        \cline{1-7}
        \end{tabular}
      \end{adjustbox}
    \end{minipage}
    \vspace*{0.8em}

    \begin{minipage}{0.05\textwidth}\vspace{0pt}
      \rotatebox{90}{\footnotesize SBS dynamics}
      \vspace*{-1em}
    \end{minipage}%
    \begin{minipage}{\tabsize}\vspace{0pt}
      \begin{adjustbox}{width=\linewidth}
        \begin{tabular}{l|cccccc|c}
        \thickhline
        \textbf{Benchmark} & \textbf{SBS} & \textbf{SMD-SBS Mean} & \textbf{SMD-SBS} $\mathbf{M^2}$ & \textbf{SMD-SBS Var} & \textbf{SMD-SBS Mean+Var} & \textbf{GCN-SBS} & \textbf{p-value} \\
        \hline
        Ackley & $21.126$ & $20.851$ & $21.112$ & $21.102$ & $20.825$ & $\cellcolor{cell-gray} \mathbf{20.810}$ & $0.000$ \\
        Deb N.1 & $-0.961$ & $-0.835$ & $-1.000$ & $-1.000$ & $\cellcolor{cell-gray} \mathbf{-1.000}$ & $-0.909$ & $0.000$ \\
        Griewank & $19.959$ & $19.686$ & $20.646$ & $20.218$ & $19.953$ & $\mathbf{19.128}$ & $0.173$ \\
        Levy & $\cellcolor{cell-gray} \mathbf{24.036}$ & $47.680$ & $25.827$ & $25.815$ & $46.601$ & $45.794$ & $0.005$ \\
        Rastrigin & $86.442$ & $185.464$ & $83.496$ & $\cellcolor{cell-gray} \mathbf{80.785}$ & $190.917$ & $181.370$ & $0.006$ \\
        Schwefel & $5986.360$ & $5984.992$ & $5989.562$ & $6062.542$ & $\mathbf{5947.981}$ & $5985.055$ & $0.365$ \\
        Styblinski-Tang & $\cellcolor{cell-gray} \mathbf{-25.668}$ & $-24.339$ & $-23.746$ & $-24.191$ & $-24.578$ & $-25.150$ & $0.008$ \\
        \cdashline{1-8}[2pt/1pt]
        Avg Rank & $3.29$ & $4.00$ & $4.29$ & $3.57$ & $3.00$ & $\mathbf{2.86}$  \\
        ECR & $15.1622$ & $15.4879$ & $15.1930$ & $15.1820$ & $\mathbf{1.3467}$ & $15.4564$ \\
        \cline{1-7}
        \end{tabular}
      \end{adjustbox}
    \end{minipage}

    \begin{minipage}{0.05\textwidth}\vspace{0pt}
      \rotatebox{90}{\footnotesize Langevin dynamics}
      \vspace*{-1em}
    \end{minipage}%
    \begin{minipage}{\tabsize}\vspace{0pt}
      \vspace*{1em}
      \begin{adjustbox}{width=\linewidth}
        \begin{tabular}{l|cccccc|c}
        \thickhline
        \textbf{Benchmark} & \textbf{Langevin} & \textbf{SMD-Langevin Mean} & \textbf{SMD-Langevin} $\mathbf{M^2}$ & \textbf{SMD-Langevin Var} & \textbf{SMD-Langevin Mean+Var} & \textbf{GCN-Langevin} & \textbf{p-value} \\
        \hline
        Ackley & $\mathbf{20.723}$ & $20.767$ & $20.732$ & $20.733$ & $20.757$ & $20.760$ & $1.000$ \\
        Deb N.1 & $-0.921$ & $-0.956$ & $-1.000$ & $\cellcolor{cell-gray} \mathbf{-1.000}$ & $\cellcolor{cell-gray} \mathbf{-1.000}$ & $-0.958$ & $0.000$ \\
        Griewank & $19.362$ & $19.117$ & $18.711$ & $19.193$ & $18.992$ & $\mathbf{18.516}$ & $0.097$ \\
        Levy & $1.329$ & $1.847$ & $1.173$ & $\cellcolor{cell-gray} \mathbf{1.073}$ & $1.904$ & $1.796$ & $0.000$ \\
        Rastrigin & $184.367$ & $189.482$ & $\mathbf{183.202}$ & $184.258$ & $193.141$ & $187.751$ & $0.677$ \\
        Schwefel & $2405.830$ & $2436.136$ & $2383.227$ & $\mathbf{2324.940}$ & $2381.518$ & $2404.611$ & $0.316$ \\
        Styblinski-Tang & $-34.902$ & $-33.186$ & $\mathbf{-35.050}$ & $-34.914$ & $-32.868$ & $-33.206$ & $0.278$ \\
        \cdashline{1-8}[2pt/1pt]
        Avg Rank & $3.86$ & $5.14$ & $\mathbf{2.00}$ & $2.14$ & $4.00$ & $3.71$  \\
        ECR & $15.1944$ & $15.3272$ & $15.1613$ & $\mathbf{1.0108}$ & $1.2014$ & $15.3117$ \\
        \cline{1-7}
        \end{tabular}
      \end{adjustbox}
    \end{minipage}

    \begin{minipage}{0.05\textwidth}\vspace{0pt}
      \rotatebox{90}{\footnotesize MSGD dynamics}
      \vspace*{-1em}
    \end{minipage}%
    \begin{minipage}{\tabsize}\vspace{0pt}
      \vspace*{1em}
      \begin{adjustbox}{width=\linewidth}
        \begin{tabular}{l|cccccc|c}
        \thickhline
        \textbf{Benchmark} & \textbf{MSGD} & \textbf{SMD-MSGD Mean} & \textbf{SMD-MSGD} $\mathbf{M^2}$ & \textbf{SMD-MSGD Var} & \textbf{SMD-MSGD Mean+Var} & \textbf{GCN-MSGD} & \textbf{p-value} \\
        \hline
        Ackley & $\mathbf{19.977}$ & $20.725$ & $19.985$ & $19.982$ & $20.736$ & $20.736$ & $0.497$ \\
        Deb N.1 & $-0.567$ & $-0.834$ & $\cellcolor{cell-gray} \mathbf{-1.000}$ & $\cellcolor{cell-gray} \mathbf{-1.000}$ & $\cellcolor{cell-gray} \mathbf{-1.000}$ & $-0.920$ & $0.000$ \\
        Griewank & $19.467$ & $\mathbf{18.829}$ & $19.557$ & $20.007$ & $19.254$ & $19.358$ & $0.243$ \\
        Levy & $9.413$ & $1.905$ & $8.873$ & $9.330$ & $1.888$ & $\cellcolor{cell-gray} \mathbf{1.354}$ & $0.000$ \\
        Rastrigin & $187.689$ & $187.206$ & $193.505$ & $186.581$ & $\mathbf{183.479}$ & $183.768$ & $0.278$ \\
        Schwefel & $2396.030$ & $2424.021$ & $\mathbf{2361.875}$ & $2386.382$ & $2380.197$ & $2447.570$ & $0.488$ \\
        Styblinski-Tang & $-36.325$ & $-34.665$ & $\cellcolor{cell-gray} \mathbf{-36.379}$ & $-36.353$ & $-34.457$ & $-34.867$ & $0.042$ \\
        \cdashline{1-8}[2pt/1pt]
        Avg Rank & $4.14$ & $3.86$ & $3.00$ & $3.14$ & $\mathbf{2.86}$ & $3.57$  \\
        ECR & $16.0058$ & $15.3008$ & $1.8064$ & $1.8554$ & $\mathbf{1.1646}$ & $15.2352$ \\
        \cline{1-7}
        \end{tabular}
      \end{adjustbox}
    \end{minipage}
\end{table*}

\section{Numerical Experiments} \label{sec:numerics}

We validate our noise injection methods on multimodal benchmark functions \cite{serre2025stein,pinnau2017cbo,Malherbe2017}, comparing standard \MKV dynamics (CBO, SBS, Langevin, MSGD) against their SMD and GCN variants. Implementation uses C++ with Eigen and Python bindings\footnote{Anonymous link upon acceptance}, with Euler-Maruyama schemes from \cref{appendix:SMD_scheme,appendix:GCN_scheme}.

\paragraph{\textbf{Computational Complexity.}} SMD incurs $\mathcal{O}(N)$ complexity per iteration ($\mathcal{O}(1)$ for mean perturbation), while GCN requires $\mathcal{O}(N^3)$ for covariance matrix square root computation, making SMD significantly more efficient for large particle systems.

Each dynamic is evaluated on $7$ multimodal benchmarks in dimension $20$ with $150$ particles over $300$ iterations, averaged over $50$ runs. Parameters: stepsize $\dd t=0.1$, SBS kernel bandwidth $1/N^2$ \cite{serre2025stein}, CBO parameters $\lambda=1$, $\epsilon=10^{-2}$, $\beta=1$, $\sigma=5.1$ \cite{carrillo2021cbo}, Langevin temperature $\beta=1$, GCN Gaussian kernel bandwidth $\sigma=1$. To assess statistical significance, we perform one or several Mann-Whitney U tests depending on which type of method achieves the best result: if a vanilla dynamic achieves the best result, we compare it against all its SMD and GCN variants; if a SMD or GCN variant achieves the best result, we compare it only against the vanilla dynamic.

Best results are in \textbf{bold}. Mann-Whitney U tests at level $0.05$ are highlighted in \textcolor{gray}{gray}. We report average rank and empirical competitive ratio (ECR), $\text{ECR}(m) = \frac{1}{|F|} \sum_{f \in F} \min (100, df_m/df^\star)$, where $df_m$ is the distance to optimum for method $m$ and $df^\star$ is the best distance.

\cref{tab:multimodal-results} shows that noise injection consistently enhances baseline performance. SMD variants generally outperform GCN, with significant gains for CBO. MSGD, Langevin and SBS also benefit from noise injection, showing reliable improvements across multiple benchmarks. \cref{fig:impact_smd} shows SMD mean perturbation improves CBO and MSGD on Ackley ($d=4$). \cref{fig:impact_gcn} demonstrates GCN's robustness to kernel bandwidth $\sigma$ on Levy ($d=10$).

\paragraph{Problem-Dependent Effects.} More broadly, \cref{tab:multimodal-results} and \cref{fig:impact_smd,fig:impact_gcn} reveal that noise injection effectiveness is intimately tied to both the baseline dynamic and problem characteristics. A critical factor is the particle-to-dimension ratio. When the number of particles is sufficiently large relative to the problem dimension, the baseline dynamics converge reliably, and the added noise provides no substantial benefit (see \cref{fig:dim-low}). In fact, it may even degrade performance by introducing unnecessary perturbations. Conversely, when particles are scarce relative to dimension, the additional exploration provided by noise injection yields modest but consistent improvements by helping particles escape local minima (see \cref{fig:dim-high}). These findings suggest that reproducing our results in higher dimensions requires careful tuning of particle counts to maintain the balance observed in our $d=20$ experiments. Specifically, one should increase the number of particles proportionally to avoid regime shifts where noise becomes detrimental. Developing adaptive strategies for automatically selecting noise injection methods and hyperparameters based on the problem's intrinsic dimensionality remains an important direction for future work.

\paragraph{Unimodal Functions.} We evaluate noise injection on a set of unimodal benchmark functions in dimension $20$ with $150$ particles over $300$ iterations. The results (detailed in \cref{tab:unimodal-results}) reveal similar trends as in the multimodal case for CBO dynamics, but more mixed outcomes for gradient-aware methods. This aligns with the intuition that noise injection primarily aids exploration in multimodal landscapes. For convex functions with a single global minimum, the added stochasticity provides limited advantage to methods that already leverage gradient information effectively, and may even degrade performance by introducing unnecessary perturbations (see \cref{fig:impact_smd_square}).

\section{Conclusion}

This work addresses a key limitation of \MKV optimization methods: deterministic mean-field dynamics that hinder exploration. Viewing \MKV systems as a unifying framework for methods such as CBO and SBS, we proposed two principled ways to inject common noise: a perturbative approach based on conditional \MKV theory and a geometric approach exploiting tangent-space structure. Treating noise as a modular component makes our framework applicable to \textit{any} \MKV algorithm.
Experiments show that both strategies improve exploration and convergence. We recommend \textbf{SMD Mean+Var} as a practical default, combining strong empirical performance with $\mathcal{O}(N)$ computational cost, compared to $\mathcal{O}(N^3)$ for GCN. Future work will focus on adaptive parameter tuning and extensions beyond global optimization.


\section*{Impact Statement}

This paper presents work whose goal is to advance the field of Optimization. There are many potential societal consequences of our work, none which we feel must be specifically highlighted here.

\bibliography{biblio.bib}
\bibliographystyle{icml2026}

\newpage
\appendix
\onecolumn
\section{Simulation scheme for Stochastic Moment Dynamics}
\label{appendix:SMD_scheme}
We use an Euler-Maruyama discretization of equation \cref{eq:particle_smd_continuous_icml} to compute the evolution of the particle system. Let $(\Delta t_n )_{n\in \mathbb{N}}$ be a sequence of positive (and possibly decreasing) time step and define the discrete times $t_n = \sum_{k=0}^n\Delta t_k$.
The Euler-Maruyama scheme associated with~\cref{eq:particle_smd_continuous_icml}
is given, for each particle $i\in [N]$, by
{\small
\begin{equation}\label{eq:scheme-SMD}
\begin{aligned}
X_{n+1}^i
&={}
X_n^i
+ \Delta t_{n+1} \cdot \left[b\!\left(X_n^i, \hat{\loi}_{X_n}\right)
+ \tilde b\!\left(X_n^i, \hat{\loi}_{X_n}\right)\right] \\
& \!\!\!
+ \sqrt{\Delta t_{n+1}} \cdot \left[ \sigma\!\left(X_n^i, \hat{\loi}_{X_n}\right)\xi_n^i
+ \tilde\sigma\!\left(X_n^i, \hat{\loi}_{X_n}\right)\zeta_n \right] ,
\end{aligned}
\end{equation}
}

where $(\xi_n^i)_{i\in [N], n\ge0}$ and $(\zeta_n)_{n\ge0}$ are independent standard Gaussian increments of respective dimension $d$ and $p$.

\section{Simulation scheme for Geometric Common Noise}
\label{appendix:GCN_scheme}
We use an Euler-Maruyama discretization of equation \cref{eq:noise_svgd} to compute the evolution of the particle system. Let $(\Delta t_n )_{n\in \mathbb{N}}$ be a sequence of positive (and possibly decreasing) time steps and define the discrete times $t_n = \sum_{k=0}^n\Delta t_k$. 
Denoting by $k_{\text{noise}}$ the kernel used for Geometric Common Noise and $k$ the kernel used for SVGD (in the presentation $k=k_{\text{noise}}$ to have a canonical meaning, but this is not necessary in general), the Euler-Maruyama scheme for the GCN-noised SVGD $N$-particle dynamics is given by:
{\small
\begin{equation}\label{eq:scheme-GCN}
\begin{aligned}
\bar{X}_{n+1}
&={}
\bar{X}_n
+ \Delta t_{n+1} \cdot \bar{b}_n
+ \sqrt{\Delta t_{n+1}} \cdot \bar{K}_{\bar{X}_n}^{1/2}\bar{\xi}_n,
\end{aligned}
\end{equation}
}
where $\bar{X}_n = (X_n^1, \ldots, X_n^N) \in \mathbb{R}^{Nd}$ is the concatenated state of all particles at time $t_n$, and
\begin{align*}
&\bar{b}_n = \left(b^{\text{svgd}}(X_n^i, \hat{\loi}_{X_n})\right)_{i\in[N]}, \quad \hat{\loi}_{X_n} = \frac{1}{N}\sum_{i=1}^N \delta_{X_n^i}, \\
&b^{\text{svgd}}(\cdot,\loi) \triangleq  \mathbb{E}_{X\sim\loi} \Big[-k(\cdot,X)\,\nabla V(X)+\nabla_x k(\cdot,X) \Big], \\
&\bar{K}_{\bar{X}_n} = K_{\bar{X}_n} \otimes I_d, \quad K_{\bar{X}_n} = (k_{\text{noise}}(X_n^i,X_n^j))_{(i,j)\in[N]^2} \in \mathbb{R}^{N\times N}, \\
&\bar{\xi}_n \sim \mathcal{N}(0_{Nd},I_{Nd}) \text{ are independent standard Gaussian increments.}
\end{align*}

\noindent For other particle dynamics, GCN noise can also be incorporated by adding the term $\sqrt{\Delta t_{n+1}} \cdot \bar{K}_{\bar{X}_n}^{1/2}\bar{\xi}_n$ to the update scheme.

\section{Background: Stochastic Differential Equations and McKean-Vlasov Theory}
\label{appendix:background}

This appendix provides a self-contained introduction to stochastic differential equations (SDEs) and \MKV dynamics, establishing the theoretical foundations underlying the methods presented in this paper.

\subsection{Stochastic Differential Equations}

A stochastic differential equation (SDE) describes the evolution of a random process by combining deterministic drift with random fluctuations \cite{oksendal2013stochastic, karatzas1991Brownian}. The canonical form of an SDE in $\mathbb{R}^d$ is:
$$
\dd X_t = b(t, X_t) \dd t + \sigma(t, X_t) \dd B_t,
$$
where $X_t \in \mathbb{R}^d$ is the state at time $t$, $b: [0,\infty) \times \mathbb{R}^d \to \mathbb{R}^d$ is the drift coefficient, $\sigma: [0,\infty) \times \mathbb{R}^d \to \mathbb{R}^{d \times m}$ is the diffusion coefficient, and $(B_t)_{t \geq 0}$ is an $m$-dimensional standard Brownian motion. The solution $X_t$ is a continuous-time stochastic process whose law evolves according to the Fokker-Planck equation:
$$
\partial_t \loi_t = -\nabla \cdot (b(t,\cdot) \loi_t) + \frac{1}{2} \sum_{i,j} \partial_{x_i x_j}^2 (a_{ij}(t,\cdot) \loi_t),
$$
where $\loi_t$ denotes the probability density of $X_t$ and $a = \sigma \sigma^T$ is the diffusion matrix \cite{pavliotis2014stochastic}.

\paragraph{Langevin dynamics.} A fundamental example is the overdamped Langevin equation \cref{eq:langevin}, which arises naturally in statistical mechanics and Bayesian sampling \cite{pavliotis2014stochastic}. For a potential function $V: \mathbb{R}^d \to \mathbb{R}$ and temperature parameter $\epsilon > 0$:
$$
\dd X_t = -\nabla V(X_t) \dd t + \sqrt{2\epsilon} \dd B_t.
$$
Under suitable regularity conditions on $V$ (e.g., smoothness and growth conditions), the law $\loi_t$ of $X_t$ converges as $t \to \infty$ to the Gibbs measure:
$$
\loi_\epsilon(\dd x) \propto \exp\left(-\frac{V(x)}{\epsilon}\right) \dd x.
$$
This convergence can be characterized through the evolution of the relative entropy (KL divergence) $\text{KL}(\loi_t \| \loi_\epsilon)$, which decreases monotonically along trajectories. The Gibbs measure concentrates near the global minima of $V$ as $\epsilon \to 0$, making Langevin dynamics a natural tool for global optimization when combined with appropriate temperature scheduling (simulated annealing) \cite{kirkpatrick1983optimization}.

\subsection{McKean-Vlasov Equations and Mean-Field Limits}

\MKV equations extend classical SDEs by allowing the drift and diffusion coefficients to depend on the probability distribution of the solution itself, creating a nonlinear feedback mechanism \cite{sznitman1991topics, carmona2018probabilistic}. A \MKV SDE takes the form:
$$
\dd X_t = b(X_t, \loi_t) \dd t + \sigma(X_t, \loi_t) \dd B_t,
$$
where $\loi_t$ denotes the law of $X_t$ at time $t$, and the coefficients $b: \mathbb{R}^d \times \mathcal{P}(\mathbb{R}^d) \to \mathbb{R}^d$ and $\sigma: \mathbb{R}^d \times \mathcal{P}(\mathbb{R}^d) \to \mathbb{R}^{d \times m}$ now depend on the entire probability measure $\loi \in \mathcal{P}(\mathbb{R}^d)$.

\paragraph{Particle approximation and propagation of chaos.} In practice, \MKV dynamics are approximated through interacting particle systems. Consider $N$ particles $(X_t^1, \ldots, X_t^N)$ evolving according to:
$$
\dd X_t^i = b(X_t^i, \hat{\loi}_{X_t}) \dd t + \sigma(X_t^i, \hat{\loi}_{X_t}) \dd B_t^i, \quad i = 1, \ldots, N,
$$
where $\hat{\loi}_{X_t} = \frac{1}{N}\sum_{i=1}^N \delta_{X_t^i}$ is the empirical measure of the particle system, and $(B_t^i)_{i=1}^N$ are independent Brownian motions. Under suitable Lipschitz conditions on the coefficients, the empirical measure $\hat{\loi}_{X_t}$ converges as $N \to \infty$ to the deterministic measure $\loi_t$ satisfying the \MKV equation. This phenomenon, known as \emph{propagation of chaos} \cite{sznitman1991topics}, ensures that particles become asymptotically independent and identically distributed according to $\loi_t$.

The mean-field limit reduces the stochastic particle system to a deterministic evolution on probability space governed by:
$$
\partial_t \loi_t = -\nabla \cdot (b(\cdot, \loi_t) \loi_t) + \frac{1}{2} \sum_{i,j} \partial_{x_i x_j}^2 (a_{ij}(\cdot, \loi_t) \loi_t),
$$
where $a(\cdot, \loi) = \sigma(\cdot, \loi) \sigma(\cdot, \loi)^T$. This is a nonlinear partial differential equation in the space of probability measures \cite{carmona2018probabilistic}.

\paragraph{Common Noise in \MKV Systems.}
In standard \MKV systems, the independent Brownian motions $(B_t^i)_{i=1}^N$ average out in the mean-field limit, making the evolution of $\loi_t$ deterministic. To maintain stochasticity at the population level, one can introduce \emph{common noise}, a Brownian motion $W_t$ shared by all particles \cite{germain2025stochastic}:
$$
\dd X_t^i = b(X_t^i, \hat{\loi}_{X_t}) \dd t + \sigma(X_t^i, \hat{\loi}_{X_t}) \dd B_t^i + \tilde{\sigma}(X_t^i, \hat{\loi}_{X_t}) \dd W_t.
$$
Unlike the independent noise $\sigma \dd B_t^i$ which vanishes as $N \to \infty$, the common noise $\tilde{\sigma} \dd W_t$ persists in the mean-field limit, inducing a stochastic evolution of $\loi_t$. This enables collective transitions between basins of attraction, enhancing exploration in multimodal landscapes.

\section{Background: Random Generalized Functions}
\label{appendix:random_generalized_function}

Many stochastic models arising in machine learning, physics, and spatial statistics involve random objects that are too irregular to be defined pointwise. A convenient and rigorous framework to handle such objects is that of \emph{random generalized functions} (also called distribution-valued random variables).

Let $\mathcal{D}$ denote a space of test functions (e.g., $C_c^\infty(\mathbb{R}^d)$ or the Schwartz space $\mathcal{S}$). A generalized function is a continuous linear functional $T : \mathcal{D} \to \mathbb{R}$. A \emph{random generalized function} $X$ is a mapping such that, for every $\varphi \in \mathcal{D}$, the quantity $X(\varphi)$ is a real-valued random variable, and the map $\varphi \mapsto X(\varphi)$ is linear and continuous (in probability).

This formulation avoids pointwise evaluation and instead characterizes randomness through its action on test functions. Classical stochastic processes that fail to exist as functions—most notably Gaussian white noise—are naturally defined in this sense. For example, white noise $W$ is characterized by
\[
\mathbb{E}[W(\varphi)W(\psi)] = \langle \varphi, \psi \rangle_{L^2},
\]
and therefore defines a centered Gaussian random generalized function.

Random generalized functions play a central role in stochastic partial differential equations, Gaussian processes with singular covariance kernels, and probabilistic models involving infinite-dimensional randomness, see \cite{da2014stochastic} to go further.

\section{Proofs}
\label{appendix:proof}

\subsection{Derivation of the coefficients $\tilde{b}$ and $\tilde{\sigma}$} \label{appendix:derivationcoeffs}

The derivation of the coefficients $\tilde b$ and $\tilde \sigma$ follows from \citet[Eq. 8]{germain2025stochastic}. This formula, stated under the assumption that the observable can be expressed as a moment of the measure:
\begin{equation} \label{eq:expressionmoment}
    F(\mu)=\mu(f):=\int_{\R^d}f(x)d\mu(x)
\end{equation}

with $f:\R^d\rightarrow \R^p$, gives the expression of $\tilde{b}$ and $\tilde \sigma$ in terms of $f$, and of the coefficients $a$ and $s$:

\begin{equation} \label{eq:formulecoeffs}
         \left\{
        \begin{array}{lll}
        \tilde \sigma_{\mu}(x) &=& \nabla f (x) \big(\mu[\nabla f^T \nabla f]\big) ^{-1} s(\mu(f))\\
    
        \tilde b_{\mu}(x) &=& \nabla f (x) \big( \mu[\nabla f^T \nabla f]\big)^{-1} [a(\mu(f))-\frac{1}{2}\mu(\tilde \sigma_{\mu}\tilde \sigma_{\mu}^T :\nabla^2 f)] \,.
        \end{array}\right.
\end{equation}

where $\sigma_{\mu}\tilde \sigma_{\mu}^T :\nabla^2 f=(\sigma_{\mu}\tilde \sigma_{\mu}^T :\nabla^2 f_1, ...,\sigma_{\mu}\tilde \sigma_{\mu}^T :\nabla^2 f_p)$ is a vector of $\R^p$, and $A:B$ represents the usual Hilbert-Schmidt scalar product between two matrices of $\R^{d\times d}$.

Because $F_\text{mean}$ and $F_{s.m}$ directly fit this setting with:
\begin{equation} 
    f_\text{mean}:x\rightarrow x , \quad f_{s.m}:x\mapsto (x_1^2,..,x_d^2)
\end{equation}

and respective coefficients $a_\text{mean}=0, \quad s_\text{mean}=I_d$, and $a_{s.o}(z)=[(\frac{\delta-1/2}{z_i})]_{i\in[d]},\quad s_{s.o}=I_d$

Equations \cref{eq:mean} and \cref{eq:moment2} giving the expression of $\tilde b$ and $\tilde \sigma$ for $F_\text{mean}$ and $F_{s.m}$ follow from a direct explicit computation of equation \cref{eq:formulecoeffs}.

Observables $F_\text{var}$ and $F_\text{mean+var}$ do not directly fit this setting so one cannot directly compute equation $\tilde b$ and $\tilde \sigma$ with \cref{eq:formulecoeffs}. To do so, first assume that the particle system defined by equation \cref{eq:controlemoment} has its observable 
$F_\text{mean+var}(\hat{\loi}_{Y_t})$ satisfying the following SDE:

\begin{equation} \label{eq:systvariance}
         \left\{
        \begin{array}{lll}
        \dd F_\text{mean}(\hat{\loi}_{Y_t}) &=& \epsilon \dd W_t^1\\
    
        \dd F_\text{var}(\hat{\loi}_{Y_t}) &=& \frac{\delta -1}{F_\text{var}(\hat{\loi}_{Y_t})}\dd t+\dd W_t^2 \,.
        \end{array}\right.
\end{equation}

where $W^1$ and $W^2$ are two independent $d$ dimensional Brownian motions and $\epsilon\in{0,1}$ corresponds respectively to the control of the Variance alone and of the mean and the variance. Again, the previous equation should be understood component by component since $F_\text{var}(\hat{\loi}_{Y_t})\in \R^d$. Using the fact that $F_\text{var}(\hat{\loi}_{Y_t})=F_{s.m}(\hat{\loi}_{Y_t})-F_\text{mean}(\hat{\loi}_{Y_t})^2$, and that because of Itô's formula:

\[\dd F_\text{mean}(\hat{\loi}_{Y_t})^2=2\epsilon F_\text{mean}(\hat{\loi}_{Y_t}) \dd W_t^1+\epsilon^2\dd t \]

equation \cref{eq:systvariance} rewrites as:

\begin{equation} 
         \left\{
        \begin{array}{lll}
        \dd F_\text{mean}(\hat{\loi}_{Y_t}) &=& \epsilon \dd W_t^1\\
    
        \dd F_{s.m}(\hat{\loi}_{Y_t}) &=& 2\epsilon F_\text{mean}(\hat{\loi}_{Y_t}) \dd W_t^1+\epsilon^2\dd t+\frac{\delta -1}{(F_{s.m}(\hat{\loi}_{Y_t})-F_\text{mean}(\hat{\loi}_{Y_t})^2)}\dd t+\dd W_t^2 \,.
        \end{array}\right.
\end{equation}

which can be written as:

\begin{equation}
    \dd \hat{\loi}_{Y_t}(f)=a(\hat{\loi}_{Y_t}(f))\dd t+ s(\hat{\loi}_{Y_t}(f)) \dd W_t
\end{equation}

with $f=(f_\text{mean}, f_{s.m}):\R^d \rightarrow \R^{2d}$, $W=(W^1,W^2)$ is a $2d$ dimensional Brownian process, and for $(z_1,z_2) \in \R^{2d}$:

\[a(z_1,z_2)=\begin{pmatrix}
0 \\
\frac{\delta+1}{2(z_2-z_1^2)}+\epsilon^2
\end{pmatrix}, \quad s(z_1,z_2)=\begin{pmatrix}
\epsilon & 0 \\
2 \epsilon z_1& 1
\end{pmatrix}\]

This formulation now fits the setting of equation \cref{eq:formulecoeffs}, which can be computed explicitly. this finally gives:

\begin{equation}
    \tilde b(x,\loi) = \left( \delta-\frac{3}{2} \right)\, \left[ \frac{x-m(\rho)}{4\Var(\rho)^2} \right], \quad \tilde \sigma(x,\loi) =\left( \epsilon I_d \, , \, \text{Diag} \left[ \frac{x-m(\rho)}{2\Var(\rho)} \right]  \right)
\end{equation}

Taking $\epsilon=1$ gives the desired result for $F_\text{mean+var}$. Taking $\epsilon=0$ (which corresponds to a non perturbation of the mean of the system), and extracting the variance component gives the result for $F_\text{var}$.

\subsection{Proof of \cref{lemma:reduced_N}} \label{appendix:lemmedensite}

\begin{proof}
Let $X=(x_1,..,x_n)\in \R^{Nd}$ such that $x_i\ne x_j$ if $i\ne j$.

Define the evaluation map
\[
\mathrm{Ev}_X:\mathcal H_d\to\mathbb R^{Nd},
\qquad
\mathrm{Ev}_X(f):=\big(f(x_1),\dots,f(x_N)\big),
\]
where $f(x_i)\in\mathbb R^d$ and the concatenation is in $\mathbb R^{Nd}$.

\paragraph{Restriction of $\mathrm{Ev}_X$ to $\mathcal H_d^X$ is an isomorphism.}
Consider the restriction $\mathrm{Ev}_X:\mathcal H_d^X\to\mathbb R^{Nd}$.
First, $\mathrm{Ev}_X$ is injective on $\mathcal H_d^X$.
Indeed, let $f\in\mathcal H_d^X$ and assume $\mathrm{Ev}_X(f)=0$. Write
$f(\cdot)=\sum_{i=1}^N K(x_i,\cdot)c_i$ for some $c=(c_1,\dots,c_N)\in\mathbb R^{Nd}$ with $c_i\in\mathbb R^d$. Then for each $l\in [N]$:

\[
0=f(x_\ell)=\sum_{i=1}^N K(x_i,x_\ell)c_i,
\]
which in block form reads $(K_X\otimes I_d)c=\bar K\,c=0$. Because $K$ is strictly positive
definite and the points $(x_i)$ are distinct, then $K_X$ is invertible, hence $\bar K$ is
invertible and $c=0$, so $f=0$. Thus $\mathrm{Ev}_X$ is injective.

Since $\dim(\mathcal H_d^X)\le Nd$ and the family $\{K(x_i,\cdot)e_j\}$ spans $\mathcal H_d^X$,
we have $\dim(\mathcal H_d^X)=Nd$. Therefore $\mathrm{Ev}_X$
is bijective. In particular, for any $g=(g_1,\dots,g_N)\in\mathbb R^{Nd}$ with $g_i\in\mathbb R^d$,
there exists a unique $f\in\mathcal H_d^X$ such that $f(X)=g$, namely
\begin{equation}
\label{eq:interpolant_inverse}
f(\cdot)=\sum_{i=1}^N K(x_i,\cdot)c_i,
\qquad
c=(K_X^{-1}\otimes I_d)\,g=\bar K^{-1}g.
\end{equation}

\paragraph{Isometry after equipping $\mathbb R^{Nd}$ with the kernel-induced norm.}
Define on $\mathbb R^{Nd}$ the inner product
\[
\langle g,h\rangle_{\bar K^{-1}} := g^\top \bar K^{-1}h,
\qquad
\|g\|_{\bar K^{-1}}^2 := g^\top \bar K^{-1}g.
\]
Then $\mathrm{Ev}_X:\mathcal H_d^X\to(\mathbb R^{Nd},\langle\cdot,\cdot\rangle_{\bar K^{-1}})$
is an isometry. Indeed, let $f\in\mathcal H_d^X$ and write $f=\sum_i K(x_i,\cdot)c_i$.
Using the product structure of $\mathcal H_d$ and the reproducing kernel calculus,
\[
\|f\|_{\mathcal H_d}^2
=
\left\langle \sum_{i}K(x_i,\cdot)c_i,\ \sum_{\ell}K(x_\ell,\cdot)c_\ell\right\rangle_{\mathcal H_d}
=
\sum_{i,\ell} c_i^\top c_\ell\,K(x_i,x_\ell)
=
c^\top (K_X\otimes I_d)c
=
c^\top \bar K\,c.
\]
On the other hand $g:=f(X)=\bar K\,c$, hence $c=\bar K^{-1}g$, and therefore
\begin{equation}
\label{eq:norm_identity}
\|f\|_{\mathcal H_d}^2
=
(\bar K^{-1}g)^\top \bar K\,(\bar K^{-1}g)
=
g^\top \bar K^{-1}g
=
\|g\|_{\bar K^{-1}}^2.
\end{equation}
This proves that $\|f\|_{\mathcal H_d}=\|\mathrm{Ev}_X(f)\|_{\bar K^{-1}}$ for all
$f\in\mathcal H_d^X$.

\paragraph{Reference (Lebesgue) measure on $\mathcal H_d^X$.}
Since $\mathcal H_d^X$ is a finite-dimensional Hilbert space, the inner product
$\langle\cdot,\cdot\rangle_{\mathcal H_d}$ induces a canonical volume measure, defined as the
Lebesgue measure in any orthonormal coordinate system. Concretely, choose an orthonormal
basis $(u_1,\dots,u_{Nd})$ of $\mathcal H_d^X$ and define, for Borel sets $A\subset\mathcal H_d^X$,
\[
\mu_{\mathcal H_d^X}(A)
:=
\lambda_{Nd}\big(\Phi(A)\big),
\qquad
\Phi(f):=\big(\langle f,u_1\rangle_{\mathcal H_d},\dots,\langle f,u_{Nd}\rangle_{\mathcal H_d}\big),
\]
where $\lambda_{Nd}$ is the usual Lebesgue measure on $\mathbb R^{Nd}$. This definition does
not depend on the chosen orthonormal basis.
Equivalently, using the isometry \cref{eq:norm_identity}, $\mu_{\mathcal H_d^X}$ is the
pullback by $\mathrm{Ev}_X$ of the Lebesgue measure on $\mathbb R^{Nd}$ associated with the
inner product $\langle\cdot,\cdot\rangle_{\bar K^{-1}}$, i.e.
\[
\mu_{\mathcal H_d^X}
=
(\mathrm{Ev}_X)^{-1}_{\#}\lambda_{\bar K^{-1}},
\]
where $\lambda_{\bar K^{-1}}$ denotes the Lebesgue measure on $\mathbb R^{Nd}$ in an orthonormal
basis for $\langle\cdot,\cdot\rangle_{\bar K^{-1}}$.
In standard Euclidean coordinates $g\in\mathbb R^{Nd}$, one has
\[
\mathrm d\lambda_{\bar K^{-1}}(g)
=
(\det \bar K)^{-1/2}\,\mathrm dg,
\]
since an orthonormalization amounts to the linear change of variables $z=\bar K^{-1/2}g$.

\paragraph{Finite-dimensional Gaussian density (proof of Lemma).}
Let $G(X)\sim\mathcal N(0,\bar K)$. Its density with respect to the standard Lebesgue measure
$\mathrm dg$ on $\mathbb R^{Nd}$ is
\[
p_{G(X)}(g)
=
\frac{1}{(2\pi)^{Nd/2}(\det\bar K)^{1/2}}
\exp\!\left(-\tfrac12\, g^\top \bar K^{-1} g\right).
\]
By \cref{eq:norm_identity}, for $f\in\mathcal H_d^X$ and $g=\mathrm{Ev}_X(f)$,
$g^\top\bar K^{-1}g=\|f\|_{\mathcal H_d}^2$. Moreover, since $\mathrm{Ev}_X$ is an isometry
between $\mathcal H_d^X$ and $(\mathbb R^{Nd},\langle\cdot,\cdot\rangle_{\bar K^{-1}})$, the
change-of-variables formula yields
\[
\mathbb P(\pi_X G\in \mathrm df)
=
p_{G(X)}(\mathrm{Ev}_X(f))\,\mathrm d\lambda_{\bar K^{-1}}(\mathrm{Ev}_X(f))
\propto
\exp\!\left(-\tfrac12\|f\|_{\mathcal H_d}^2\right)\,\mathrm d\mu_{\mathcal H_d^X}(f),
\]
which is the claimed statement.

\end{proof}

\subsection{Proof of \Cref{lemma:SDE_to_consider}}

Start from the following equation:

\begin{equation} \label{eq:iterationrho}
    \hat{\loi}_{(n+1)\epsilon}=(id+\epsilon v_{n\epsilon}+\sqrt{\epsilon}G_{n})\#\hat{\loi}_{n\epsilon},\quad G_n \sim \mathcal{G}(\beta^2 K_{SVGD})
\end{equation}

This equation can be interpreted as a noisy version of the discretized in time continuity equation for the SVGD metric.

\[\partial_t\rho_t+\nabla \cdot(\rho_t v_t)=0\]

As it is classical within PDE theory, evolutions of this type are usually discretized in space thanks to the use of particle systems.  First approximate $\hat{\rho}_0$ by the empirical distribution of a particle system $\bar{\rho}_0^\epsilon:=\frac{1}{N}\sum_{i=1}^N\delta_{X_0^{\epsilon,i}}$. Then, the iteration \cref{eq:iterationrho} initialized with $\bar \rho_0$ rewrites as an iteration on the particles:

\begin{equation}  \label{eq:dynamiquediscrtiseetempsrandom}
    X_{n+1}^{\epsilon,i}=\epsilon v_n(X_{n}^{\epsilon,i})+\sqrt{\epsilon}G_{n}(X_{n}^{\epsilon,i})
\end{equation}

The iteration \cref{eq:iterationrho} initialized with $\hat \rho_{0}$ is then approximated with:

\[\hat \rho_{n\epsilon}\approx \bar \rho_{n\epsilon}=\frac{1}{N}\sum_{i=1}^N\delta_{X_n^{\epsilon,i}}\]

We now want to derive a continuous time version of \cref{eq:dynamiquediscrtiseetempsrandom}, obtained when $\epsilon\rightarrow 0$. To do so, remark that the vector $G_n(X_n^\epsilon):=(G_{n}(X_{n}^{\epsilon,1}), ..., G_{n}(X_{n}^{\epsilon,N}))\in \R^{Nd}$ is Gaussian, and its covariance matrix is given by:

\[C_\beta(X_n^\epsilon)=\beta^2
\bar K(X_n^\epsilon), \quad K(X_n^\epsilon) \;=\;
\begin{pmatrix}
K_{1,1} &  \cdots & K_{1,N} \\
\vdots &   \ddots & \vdots \\
K_{N,1} &  \cdots & K_{N,N}
\end{pmatrix},
\qquad
K_{i,j}=k(x_i,x_j) I_d.
\]

which can be reformulated as :

\begin{equation}
    \bar K(X)=  K(X)\otimes I_d, \quad  K(X)_{i,j}=k(X^i,X^j) 
\end{equation}

Then, using the theory of Gaussian vectors, \cref{eq:dynamiquediscrtiseetempsrandom} can be expressed in a matrix form by:

\begin{equation} \label{eq:edsdiscretisée}
    X_{n+1}^{\epsilon}=\epsilon v_n(X_{n}^{\epsilon})+\sqrt{\epsilon \beta^2\bar K(X_{n}^{\epsilon})} \xi_{n+1}
\end{equation}
where $(\xi_{n})_{n\in \mathbb{N}}$ are i.i.d standard Gaussian variables of dimension $Nd$.

One can identify equation \cref{eq:edsdiscretisée} as Euler-Maruyama discretization of the following SDE:

\begin{equation} \label{eq:edssystparticules}
    dX_t=v_t(X_{t})\dd t+\beta\sqrt{\bar K(X_t)} \dd B_t
\end{equation}

The equations \cref{eq:edssystparticules}, \cref{eq:dynamiquediscrtiseetempsrandom}  and\cref{eq:edsdiscretisée} correspond respectively to to the continuous in time/ discretized in space; continuous in space/ discretized in time and discretized in time/ discretized in space version of the underlying continuous in time/ continuous in space dynamic \cref{eq:spde_noisy}.

\subsection{Proof of \cref{theorem:link}}
\begin{proof}
We have that for any $X\in \Rset^{Nd}$, $K_X\to \bar{J}$ where $\bar{J}=J\otimes I_d$ where $J=(1)_{i,j\in [N^2]}$ is the matrix with only 1 values, thus $K_X^{1/2}\to \bar{J}/\sqrt{N}$ since $\bar{J}^2=N\bar{J}$. Therefore in the case $\sigma \to \infty$, we have $k_\sigma(x,y)=1$, then by \cref{eq:G_t_N_bar}, $\bar{G}_t^N=(\sum_{i=1}^N B_t^i/\sqrt{N})_{j\in [N]}$  meaning that $G_t(X_j^N)=Z$ for any $j\in[N]$ with $Z \sim \mathcal{N}(0,I_d)$. In other words, when $\sigma \to \infty $, GCN is SMD-mean.

In the case, $\sigma \to 0$, $k_\sigma(x,y)^{(d)}\to \delta_0(x-y)$, meaning that $G\sim \mathcal{G}(k)$ satisfies $\mathbb{E}(G(x)G(y))=\delta_0(x-y)$ which defines a white noise.
\end{proof}
\subsection{Proof of \cref{thm:limit-theorem}}
We take the same starting point as the proof of Proposition 2 in \cite{duncan2023geometry}.
Let $\phi \in C_c^{\infty}\left([0, \infty) \times \mathbb{R}^d\right)$ be a smooth test function with compact support and define $\Phi \in C_c^{\infty}\left([0, \infty) \times \mathbb{R}^{N d}\right)$ by $\Phi(t, x):=\frac{1}{N} \sum_{i=1}^N \phi\left(t, x_i\right)$ for any $x=(x_1,\ldots,x_N)\in \Rset^{Nd}$. Denoting by for any $\tilde{x}\in \Rset^d$ and $\rho \in \mathcal{P}(\Rset^d)$,

$$
b(\tilde{x}, \rho)=\mathbb{E}_{X\sim\loi}
\Big[
-k(\tilde{x},X)\,\nabla V (X)+\nabla_x k(\tilde{x},X)
\Big]=\int_{\mathbb{R}^d}\left[-k(\tilde{x}, y) \nabla V(y)+\nabla_y k(\tilde{x}, x)\right] \mathrm{d} \rho(y)
$$

Let $(\bar{X}_t)_{t\geq 0}$ a solution of \cref{eq:noise_svgd}, Itô's formula implies

$$
\begin{aligned}
\mathrm{d} \Phi\left(t, \bar{X}_t\right)= & \frac{1}{N} \sum_{i=1}^N\left(\partial_t \phi\left(t, X_t^i\right)+\nabla \phi\left(t, X_t^i\right) \cdot b\left(X_t^i, \rho_t^N\right)\right) \mathrm{d} t+ \beta^2\operatorname{Tr}\left(\bar{K}_{\bar{X}_t} \operatorname{Hess} \Phi\left(t,\bar{X}_t\right)\right) \mathrm{d} t \\
& +\frac{\beta}{N} \sum_{i, j=1}^N \nabla \phi\left(t,X_t^i\right) \cdot \left[\sqrt{\bar{K}_{\bar{X}_t} }\right]_{i j} \mathrm{~d} W_t^j
\end{aligned}
$$
where $ \left[\sqrt{\bar{K}_{\bar{X}_t} }\right]_{i j}\in \Rset^{d\times d}$ is the matrix block in line $i$ and column $j$ in $\sqrt{\bar{K}_{\bar{X}_t} }$ and $(W_t=(W_t^1,\ldots, W_t^N))_{t\geq 0}$ is a $Nd$-dimensional Brownian motion.
The Hessian Hess $\Phi \in \mathbb{R}^{N d \times N d}$ consists of $N^2$ blocks of size $d \times d$ with

$$
[\operatorname{Hess} \Phi(t,x)]_{i j}=\left\{\begin{array}{ll}
\frac{1}{N} \operatorname{Hess} \phi\left(t,x_i\right) & \text { if } i=j \\
0 & \text { otherwise }
\end{array}, \quad(i, j) \in\{1, \ldots, N\}^2\right.
$$

so that it is a block diagonal matrix, with each diagonal block containing the Hessian of $\phi$.
A simple calculation yields that

$$
\operatorname{Tr}(\bar{K}_{x} \operatorname{Hess} \Phi(t,x))=\frac{1}{N} \sum_{i=1}^N k\left(x_i, x_i\right) \operatorname{Tr}\left(\operatorname{Hess} \phi\left(t,x_i\right)\right)=\frac{1}{N} \sum_{i=1}^N k\left(x_i, x_i\right) \Delta_{x_i} \phi\left(t,x_i\right),
$$

so that

$$
\operatorname{Tr}\left(\mathcal{K}\left(\bar{X}_t\right) \operatorname{Hess} \Phi\left(t,\bar{X}_t\right)\right)= \int_{\mathbb{R}^d} k(x, x) \Delta_x \phi(t,x) \mathrm{d} \rho_t^N(x)
$$

Denoting by $\rho_s^N=\sum_{i=1}^N \delta_{X_s^i}/N$, it follows that

$$
\begin{aligned}
\left\langle\phi(t, \cdot), \rho_t^N\right\rangle-\left\langle\phi(0, \cdot), \rho_0^N\right\rangle= & \int_0^t\left\langle\partial_s \phi(s, \cdot), \rho_s^N\right\rangle \mathrm{d} s+\int_0^t\left\langle\nabla \phi(s, \cdot) \cdot b\left(\cdot, \rho_s^N\right), \rho_s^N\right\rangle \mathrm{d} s \\
& + \beta^2\int_0^t\left\langle k(\cdot, \cdot) \Delta_x \phi(s,\cdot), \rho_s^N\right\rangle \mathrm{d} s+\beta\mathcal{N}_t
\end{aligned}
$$
where the brackets denote the duality pairing between test functions and measures.
The term $\mathcal{N}_t$ represents a local martingale with quadratic variation

$$
\begin{aligned}
{[\mathcal{N} ., \mathcal{N} .]_t } & =\frac{1}{N^2} \int_0^t \sum_{i, j=1}^N \nabla \phi\left(s,X_s^i\right) \cdot \mathcal{K}\left(\bar{X}_s\right)_{i j} \nabla \phi\left(s,X_s^j\right) \mathrm{d} s \\
& =\frac{1}{N^2} \int_0^t \sum_{i, j=1}^N \nabla \phi\left(s,X_s^i\right) \cdot \nabla \phi\left(s,X_s^j\right) k\left(X_s^i, X_s^j\right) \mathrm{d} s \\
& = \int_0^t \int_{\mathbb{R}^d} \int_{\mathbb{R}^d} \nabla \phi(s,y) \cdot \nabla \phi(s,z) k(y, z) \mathrm{d} \rho_s^N(y) \mathrm{d} \rho_s^N(z) \mathrm{d} s
\end{aligned}
$$

In particular, assuming that the family $\left\{\rho^N_\cdot: N \in \mathbb{N}\right\}$ possesses a limit point in $\mathcal{P}(C[0, T])$ denoted by $\rho^*_\cdot$, it follows that 
$$
    [\mathcal{N}^* ., \mathcal{N}^* .]_t=  \int_0^t \int_{\mathbb{R}^d} \int_{\mathbb{R}^d} \nabla \phi(s,y) \cdot \nabla \phi(s,z) k(y, z) \mathrm{d} \rho_s^*(y) \mathrm{d} \rho_s^*(z) \mathrm{d} s.
$$
$\mathcal{N}^*$ is uniquely defined by its quadratic variation, such that we have
\begin{equation}
    \mathcal{N}^*_t=\int_0^t \langle \nabla \phi(s,\cdot)\cdot \dd G(s,\cdot),\rho_s^*\rangle 
\end{equation}
where $\ G\sim \mathcal{TG}(K_{\text{SVGD}})$.
At the limit,
\begin{align*}
\left\langle\phi(t, \cdot), \rho_t^*\right\rangle-\left\langle\phi(0, \cdot), \rho_0^*\right\rangle= & \int_0^t\left\langle\partial_s \phi(s, \cdot), \rho_s^*\right\rangle \mathrm{d} s+\int_0^t\left\langle\nabla \phi(s, \cdot) \cdot b\left(\cdot, \rho_s^*\right), \rho_s^*\right\rangle \mathrm{d} s \\
& + \beta^2\int_0^t\left\langle k(\cdot, \cdot) \Delta_x \phi(s,\cdot), \rho_s^*\right\rangle \mathrm{d} s+ \beta\mathcal{N}_t^* 
\end{align*}
$\rho_s^*$ is thus solution of the following equation in the weak sense
\begin{equation}
    \partial_t \rho_s^* +\divergence \left(\left(b(\cdot,\rho_s^*) +\beta \dd G(s,\cdot)\right)\rho_s^* \right)  +\beta^2 \Delta_s (k(\cdot,\cdot) \rho_s^*) =0, \quad G\sim \mathcal{TG}(K_{\text{SVGD}})
\end{equation}
This concludes the proof for \cref{eq:spde_noisy}.

Regarding \cref{eq:schem_info_svgd}, the proof proceeds from the following computation
\begin{equation}
\int_0^\epsilon \dd G_i(s,\cdot)=\sum_{\vec{n}\in \Nset^d} \int_0^\epsilon \lambda_{\vec{n}}\phi_{\vec{n}}\dd \xi_{\vec{n},t}^i=\sum_{\vec{n}\in \Nset^d}  \lambda_{\vec{n}}\phi_{\vec{n}} \int_0^\epsilon \dd \xi_{\vec{n},t}^i=\sqrt{\epsilon} \sum_{\vec{n}\in \Nset^d} \lambda_{\vec{n}}\phi_{\vec{n}}  \xi_{\vec{n}}^i=\sqrt{\epsilon} g, \quad \xi_{\vec{n}}^i\overset{i.i.d}{\sim} \mathcal{N}(0,1)
\end{equation}
where $g\sim \mathcal{G}(K_{\text{SVGD}})$ by \cref{remark:decomposition_gcn} and by \cref{eq:discretization_eq_con}.



\section{Additional Numerical Experiments}

\begin{figure}[H]
\centering
\begin{subfigure}{0.45\textwidth}
    \includegraphics[width=\textwidth]{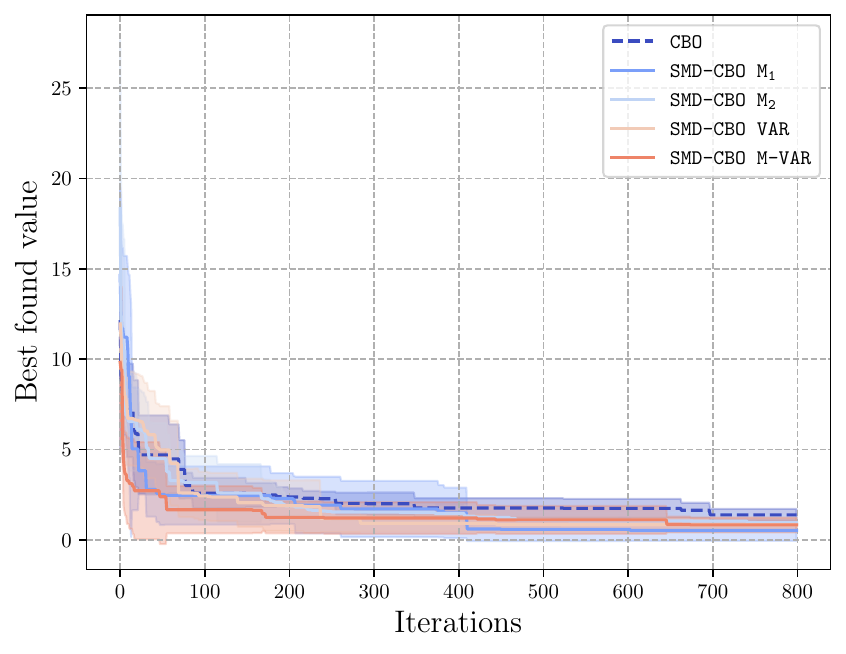}
    \caption{Impact of SMD on CBO dynamics on Rastrigin function in dimension $2$ with $10$ particles.}
    \label{fig:dim-low}
\end{subfigure}
\hfill
\begin{subfigure}{0.45\textwidth}
    \includegraphics[width=\textwidth]{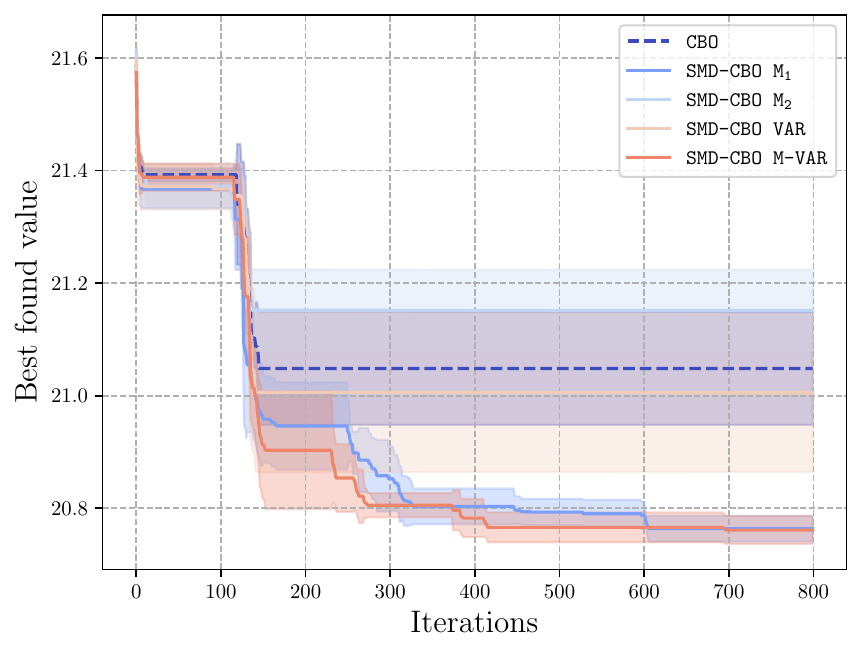}
    \caption{Impact of SMD on CBO dynamics on Ackley function in dimension $200$ with $10$ particles.}
    \label{fig:dim-high}
\end{subfigure}
  \caption{Illustration of the problem-dependent effects of noise injection methods in the multimodal setting. The mean (line) and standard deviation (shaded area) of the best found value are reported over $5$ runs. In low dimension with sufficient particles (left), the effect of SMD is negligible, while in very high dimension with scarce particles (right), SMD yields modest improvements.}
\end{figure}

\begin{figure}[h]
\centering
\begin{subfigure}{0.45\textwidth}
    \includegraphics[width=\textwidth]{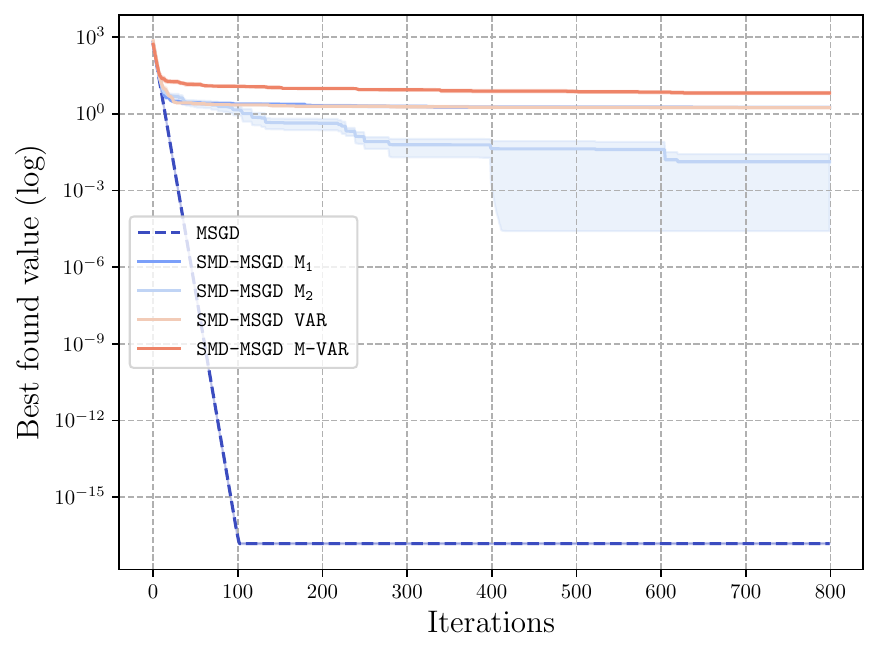}
    \caption{MSGD dynamics}
\end{subfigure}
\hfill
\begin{subfigure}{0.45\textwidth}
    \includegraphics[width=\textwidth]{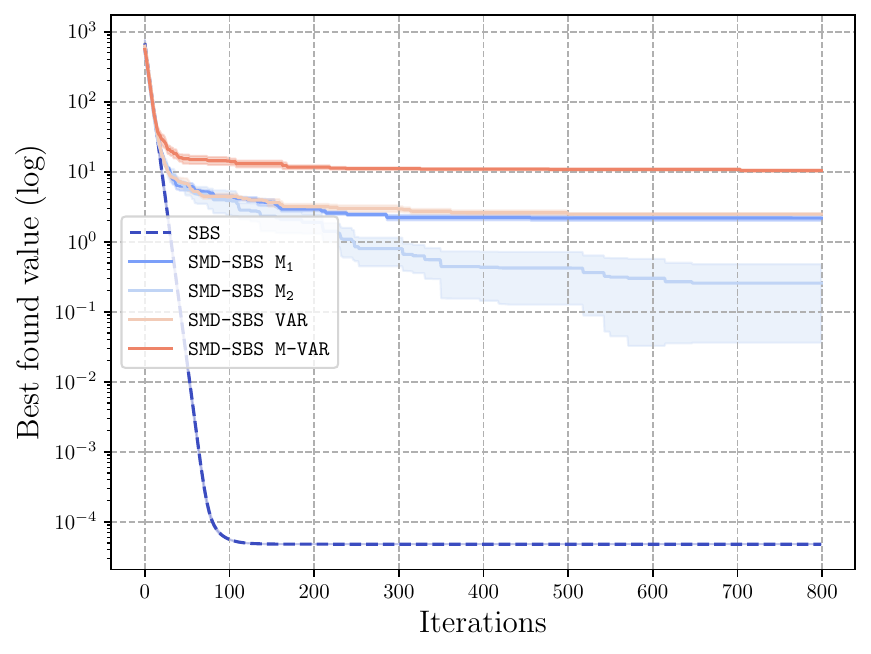}
    \caption{SBS dynamics}
\end{subfigure}
\caption{Impact of perturbed observables on MSGD (left) and SBS dynamics (right) on the Square function in dimension $20$ with $2$ particles. The mean (line) and standard deviation (shaded area) of the best found value are reported over $5$ runs. In this unimodal setting, the noise injection slightly degrades performance by introducing unnecessary perturbations.}
\label{fig:impact_smd_square}
\end{figure}

\begin{table*}[h]
  \caption{Comparison between vanilla dynamics and their SMD and GCN variants on $8$ unimodal benchmark functions in dimension $20$ using $150$ particles over $300$ iterations. For each dynamic, the average best result is reported. The average rank over all benchmarks along with the ECR is reported at the bottom of each table. The maximum p-value of Mann-Whitney U tests is reported in the last column. The best results are in bold and cells where the Mann-Whitney U test is significant at level $0.05$ are highlighted in gray.}
  \label{tab:unimodal-results}
    \centering
    \begin{minipage}{0.05\textwidth}\vspace{0pt}
      \rotatebox{90}{\footnotesize CBO dynamics}
      \vspace*{-1em}
    \end{minipage}%
    \begin{minipage}{\tabsize}\vspace{0pt}
      \begin{adjustbox}{width=\columnwidth}
        \begin{tabular}{l|cccccc|c}
        \thickhline
        \textbf{Benchmark} & \textbf{CBO} & \textbf{SMD-CBO Mean} & \textbf{SMD-CBO} $\mathbf{M^2}$ & \textbf{SMD-CBO Var} & \textbf{SMD-CBO Mean+Var} & \textbf{GCN-CBO} & \textbf{p-value} \\
        \hline
        Bent cigar & $30623584067.892$ & $\mathbf{29682175177.256}$ & $31818328261.401$ & $31057209405.099$ & $30240038880.044$ & $30457405283.166$ & $0.281$ \\
        Dixon-Price & $469983.927$ & $\cellcolor{cell-gray} \mathbf{371918.118}$ & $442706.838$ & $461788.187$ & $380525.484$ & $378987.734$ & $0.000$ \\
        Hyper-Ellipsoid & $814.670$ & $\cellcolor{cell-gray} \mathbf{714.377}$ & $806.037$ & $765.882$ & $734.561$ & $729.761$ & $0.000$ \\
        Rosenbrock & $488088.345$ & $378862.224$ & $498337.068$ & $473975.528$ & $371117.547$ & $\cellcolor{cell-gray} \mathbf{351728.228}$ & $0.000$ \\
        Square & $336.981$ & $306.784$ & $335.871$ & $317.282$ & $304.169$ & $\cellcolor{cell-gray} \mathbf{301.648}$ & $0.001$ \\
        Sumpow & $41.136$ & $\cellcolor{cell-gray} \mathbf{9.806}$ & $34.587$ & $30.481$ & $10.263$ & $10.263$ & $0.000$ \\
        Trid & $-0.005$ & $-0.006$ & $-0.005$ & $-0.004$ & $-0.005$ & $\cellcolor{cell-gray} \mathbf{-0.006}$ & $0.022$ \\
        Zakharov & $635.966$ & $\cellcolor{cell-gray} \mathbf{585.580}$ & $664.924$ & $597.844$ & $605.961$ & $606.866$ & $0.011$ \\
        \cdashline{1-8}[2pt/1pt]
        Avg Rank & $5.25$ & $\mathbf{1.62}$ & $5.25$ & $4.25$ & $2.50$ & $2.12$  \\
        ECR & $1.5283$ & $\mathbf{1.0119}$ & $1.4488$ & $1.3621$ & $1.0271$ & $1.0187$ \\
        \cline{1-7}
        \end{tabular}
      \end{adjustbox}
    \end{minipage}
    \vspace*{0.8em}
    
    \begin{minipage}{0.05\textwidth}\vspace{0pt}
      \rotatebox{90}{\footnotesize SBS dynamics}
      \vspace*{-1em}
    \end{minipage}%
    \begin{minipage}{\tabsize}\vspace{0pt}
      \begin{adjustbox}{width=\columnwidth}
        \begin{tabular}{l|cccccc|c}
        \thickhline
        \textbf{Benchmark} & \textbf{SBS} & \textbf{SMD-SBS Mean} & \textbf{SMD-SBS} $\mathbf{M^2}$ & \textbf{SMD-SBS Var} & \textbf{SMD-SBS Mean+Var} & \textbf{GCN-SBS} & \textbf{p-value} \\
        \hline
        Bent cigar & $32005162588.056$ & $31363252573.409$ & $31743875639.769$ & $31890761748.545$ & $31987545635.464$ & $\mathbf{31146357732.286}$ & $0.206$ \\
        Dixon-Price & $444531.288$ & $438354.660$ & $455444.623$ & $459341.343$ & $\mathbf{434068.195}$ & $456453.507$ & $0.627$ \\
        Hyper-Ellipsoid & $\mathbf{2.230}$ & $202.142$ & $2.725$ & $5.652$ & $207.690$ & $130.216$ & $0.570$ \\
        Rosenbrock & $\mathbf{444772.184}$ & $453814.436$ & $513107.614$ & $487215.524$ & $463959.093$ & $502791.601$ & $0.879$ \\
        Square & $153.365$ & $193.608$ & $152.775$ & $\mathbf{152.288}$ & $175.262$ & $175.231$ & $0.975$ \\
        Sumpow & $0.260$ & $1.743$ & $0.002$ & $\cellcolor{cell-gray} \mathbf{7e^{-4}}$ & $0.039$ & $0.160$ & $0.000$ \\
        Trid & $0.004$ & $0.001$ & $0.002$ & $0.002$ & $0.001$ & $\cellcolor{cell-gray} \mathbf{-8e^{-5}}$ & $0.000$ \\
        Zakharov & $1018737.046$ & $450180.203$ & $378128.839$ & $612121.643$ & $734513.842$ & $\mathbf{62513.402}$ & $0.293$ \\
        \cdashline{1-8}[2pt/1pt]
        Avg Rank & $3.88$ & $3.62$ & $3.25$ & $3.38$ & $3.75$ & $\mathbf{3.12}$  \\
        ECR & $15.2971$ & $25.3978$ & $\mathbf{1.8410}$ & $2.3145$ & $20.2180$ & $20.5921$ \\
        \cline{1-7}
        \end{tabular}
      \end{adjustbox}
    \end{minipage}
    \vspace*{0.8em}

    \begin{minipage}{0.05\textwidth}\vspace{0pt}
      \rotatebox{90}{\footnotesize Langevin dynamics}
      \vspace*{-1em}
    \end{minipage}%
    \begin{minipage}{\tabsize}\vspace{0pt}
      \begin{adjustbox}{width=\columnwidth}
        \begin{tabular}{l|cccccc|c}
        \thickhline
        \textbf{Benchmark} & \textbf{Langevin} & \textbf{SMD-Langevin Mean} & \textbf{SMD-Langevin} $\mathbf{M^2}$ & \textbf{SMD-Langevin Var} & \textbf{SMD-Langevin Mean+Var} & \textbf{GCN-Langevin} & \textbf{p-value} \\
        \hline
        Bent cigar & $32535447108.573$ & $\mathbf{31621953123.250}$ & $32072428247.128$ & $31976284652.507$ & $32565581749.156$ & $31819674299.149$ & $0.224$ \\
        Dixon-Price & $451898.516$ & $454002.991$ & $468710.179$ & $460187.324$ & $\mathbf{443581.877}$ & $481332.550$ & $0.603$ \\
        Hyper-Ellipsoid & $828.543$ & $825.205$ & $\cellcolor{cell-gray} \mathbf{779.643}$ & $798.520$ & $812.076$ & $803.315$ & $0.033$ \\
        Rosenbrock & $490079.993$ & $\mathbf{459736.311}$ & $482284.302$ & $464262.042$ & $493496.166$ & $494916.053$ & $0.372$ \\
        Square & $\cellcolor{cell-gray} \mathbf{0.979}$ & $1.877$ & $1.376$ & $1.367$ & $2.545$ & $1.979$ & $0.000$ \\
        Sumpow & $6e^{-5}$ & $3e^{-5}$ & $1e^{-4}$ & $\cellcolor{cell-gray} \mathbf{1e^{-6}}$ & $2e^{-5}$ & $5e^{-5}$ & $0.000$ \\
        Trid & $-2e^{-4}$ & $1e^{-4}$ & $-5e^{-4}$ & $-2e^{-4}$ & $-3e^{-4}$ & $\cellcolor{cell-gray} \mathbf{-7e^{-4}}$ & $0.038$ \\
        Zakharov & $\mathbf{109659.961}$ & $544154.097$ & $406593.891$ & $394359.838$ & $920400.543$ & $788568.950$ & $1.000$ \\
        \cdashline{1-8}[2pt/1pt]
        Avg Rank & $3.62$ & $3.38$ & $3.38$ & $\mathbf{2.50}$ & $4.12$ & $4.00$  \\
        ECR & $6.2152$ & $4.4340$ & $10.4760$ & $\mathbf{1.3848}$ & $4.3054$ & $6.5461$ \\
        \cline{1-7}
        \end{tabular}
      \end{adjustbox}
    \end{minipage}
    \vspace*{0.8em}

    \begin{minipage}{0.05\textwidth}\vspace{0pt}
      \rotatebox{90}{\footnotesize MSGD dynamics}
      \vspace*{-1em}
    \end{minipage}%
    \begin{minipage}{\tabsize}\vspace{0pt}
      \begin{adjustbox}{width=\columnwidth}
        \begin{tabular}{l|cccccc|c}
        \thickhline
        \textbf{Benchmark} & \textbf{MSGD} & \textbf{SMD-MSGD Mean} & \textbf{SMD-MSGD} $\mathbf{M^2}$ & \textbf{SMD-MSGD Var} & \textbf{SMD-MSGD Mean+Var} & \textbf{GCN-MSGD} & \textbf{p-value} \\
        \hline
        Bent cigar & $31678046302.237$ & $31462840710.309$ & $31473260670.919$ & $32549106519.774$ & $32133802466.087$ & $\mathbf{30786379896.381}$ & $0.216$ \\
        Dixon-Price & $\mathbf{440543.067}$ & $453061.646$ & $459649.761$ & $444971.192$ & $450100.101$ & $453917.746$ & $1.000$ \\
        Hyper-Ellipsoid & $794.589$ & $785.260$ & $798.768$ & $789.540$ & $801.574$ & $\mathbf{773.490}$ & $0.471$ \\
        Rosenbrock & $501435.007$ & $474768.154$ & $\mathbf{458605.030}$ & $481823.429$ & $470479.892$ & $480760.075$ & $0.155$ \\
        Square & $\cellcolor{cell-gray} \mathbf{3e^{-18}}$ & $1.789$ & $5e^{-11}$ & $0.005$ & $1.549$ & $1.000$ & $0.000$ \\
        Sumpow & $3e^{-5}$ & $0.012$ & $\cellcolor{cell-gray} \mathbf{0}$ & $\cellcolor{cell-gray} \mathbf{0}$ & $\cellcolor{cell-gray} \mathbf{0}$ & $9e^{-5}$ & $0.000$ \\
        Trid & $0.003$ & $0.001$ & $0.002$ & $0.002$ & $9e^{-4}$ & $\cellcolor{cell-gray} \mathbf{-3e^{-4}}$ & $0.000$ \\
        Zakharov & $555551.173$ & $\cellcolor{cell-gray} \mathbf{253587.981}$ & $508868.077$ & $395610.144$ & $306451.213$ & $450211.111$ & $0.030$ \\
        \cdashline{1-8}[2pt/1pt]
        Avg Rank & $4.00$ & $3.38$ & $3.50$ & $3.38$ & $3.25$ & $\mathbf{3.12}$  \\
        ECR & $13.5453$ & $25.7640$ & $\mathbf{1.2142}$ & $13.4638$ & $13.4180$ & $25.8568$ \\
        \cline{1-7}
        \end{tabular}
      \end{adjustbox}
    \end{minipage}
\end{table*}


\end{document}